\documentclass[A4]{amsart}

\input xypic
\usepackage{amssymb}
\usepackage{MnSymbol}
\usepackage{bbm}
\usepackage{amsmath}
\newtheorem{theorem}{Theorem}[section]

\newtheorem{proposition}[theorem]{Proposition}
\newtheorem{lemma}[theorem]{Lemma}
\newtheorem{corollary}[theorem]{Corollary}

\newtheorem{remark}[theorem]{Remark}

\newcommand{\Keywords}[1]{{\footnotesize\noindent\textit{Keywords:}
  \parbox[t]{120mm}{\raggedright\footnotesize#1}}\vspace{.5pc}}

\newcommand{\amsprimary}[1]{{\footnotesize\noindent AMS 2010 \textit{Mathematics subject
classification:} Primary #1\vspace{1pc}}}

\long\def\symbolfootnote[#1]#2{\begingroup
\def\thefootnote{\fnsymbol{footnote}}\footnote[#1]{#2}\endgroup}

\newcounter{bean}

\newenvironment{romanlist}{\begin{list}{\rm ({\roman{bean}})}
      {\usecounter{bean}\setlength{\rightmargin}{\leftmargin}}}
      {\end{list}}

\newcommand{\seqm}[3]{\ensuremath{#1\stackrel{#2}
 {\longrightarrow}#3}}
\newcommand{\seqmm}[5]{\ensuremath{#1\stackrel{#2}
 {\longrightarrow}#3\stackrel{#4}{\longrightarrow}#5}}
\newcommand{\seqmmm}[7]{\ensuremath{#1\stackrel{#2}
 {\longrightarrow}#3\stackrel{#4}{\longrightarrow}#5
  \stackrel{#6}{\longrightarrow}#7}}
\newcommand{\seqmmmm}[9]{\ensuremath{#1\stackrel{#2}
 {\longrightarrow}#3\stackrel{#4}{\longrightarrow}#5
  \stackrel{#6}{\longrightarrow}#7
  \stackrel{#8}{\longrightarrow}#9}}

\newcommand{\floor}[1]{\ensuremath{\left\lfloor #1 \right\rfloor}}
\newcommand{\ceil}[1]{\ensuremath{\left\lceil #1 \right\rceil}}
\newcommand{\paren}[1]{\ensuremath{\left\{ #1 \right\}}}
\newcommand{\bracket}[1]{\ensuremath{\left( #1 \right)}}
\newcommand{\vbracket}[1]{\ensuremath{\left\langle#1\right\rangle}}

\newcommand{\cvee}[3]{\displaystyle\bigvee^{#2}_{#1}#3}
\newcommand{\cprod}[3]{\displaystyle\prod^{#2}_{#1}#3}

\newcommand{\csum}[3]{\displaystyle\sum^{#2}_{#1}#3}
\newcommand{\cunion}[3]{\displaystyle\bigcup^{#2}_{#1}#3}

\newcommand{\qqed}{\hfill\square}

\newcommand{\zmodp}{\ensuremath{\mathbb{Z}_{p}}}
\newcommand{\zmod}[1]{\ensuremath{\mathbb{Z}_{p^{#1}}}}
\newcommand{\zmodtwo}{\ensuremath{\mathbb{Z}_{2}}}

\newcommand{\M}[2]{\ensuremath{P^{#1}(p^{#2})}}
\newcommand{\MS}[2]{\ensuremath{S^{#1}\{p^{#2}\}}}

\def\dim{\mathrm{dim}}

\begin{document}

%%% Title
\title{The homotopy type of a Poincar\'e Duality Complex after Looping}

\author{Piotr Beben}
\address{\scriptsize{School of Mathematics, University of Southampton,Southampton SO17 1BJ, United Kingdom}} 
\email{P.D.Beben@soton.ac.uk} 

\author{Jie Wu} 
\address{\scriptsize{Department of Mathematics, National University of Singapore,
Block S17 (SOC1),
10, Lower Kent Ridge Road,
Singapore 119076}}  
\email{matwuj@nus.edu.sg}

\begin{abstract}
We answer a weaker version of the classification problem 
for the homotopy types of $(n-2)$-connected closed orientable $(2n-1)$-manifolds.   
Let $n\geq 6$ be an even integer, and 
$X$ be a $(n-2)$-connected finite orientable Poincar\'e $(2n-1)$-complex such that
$H^{n-1}(X;\mathbb{Q})=0$ and $H^{n-1}(X;\zmodtwo)=0$.
Then its loop space homotopy type is uniquely determined by the action of 
higher Bockstein operations on $H^{n-1}(X;\zmodp)$ for each odd prime $p$.
A stronger result is obtained when localized at odd primes.

\end{abstract}

\maketitle

\symbolfootnote[0]{\amsprimary{55P35; 55P15; 57N65.}}
\symbolfootnote[0]{\Keywords{Poincar\'e duality complex, loop space, homotopy classification of manifolds.}}

\section{Introduction}

A connected space $X$ is said to satisfy Poincar\'e duality with respect to a coefficient ring $R$ if 
for some fixed nonzero class $e\in H_{n}(X;R)$ the cap product 
$$
\seqm{e\cap H^{i}(X;R)}{}{H_{n-i}(X;R)}
$$ 
pairing is an isomorphism for each $0\leq i\leq n$. These isomorphisms put additional restraints on the cohomology ring. 
If we fix $R$ to be a field, then $H^{i}(X;R)=0$ for $i>n$, and
$H^{n-i}(X;R)\cong H^{i}(X;R)$ for $0\leq i\leq n$. In particular, $H^{n}(X;R)\cong H^{0}(X;R)\cong R$,
and there are nonsingular cup product pairings
$$
\seqm{H^{i}(X;R)\otimes H^{n-i}(X;R)}{\cup}{H^{n}(X;R)\cong R},
$$
so the maps $\seqm{H^{n-i}(X;R)}{}{Hom(H^{i}(X;R),R)}$ and
$\seqm{H^{i}(X;R)}{}{Hom(H^{n-i}(X;R),R)}$ induced by the above pairing are isomorphisms. 
Every nonzero element $x\in H^{n-i}(X;R)$ then corresponds to a nonzeo element 
$y\in H^{i}(X;R)$ such that the cup product $xy$ is nonzero in $H^{n}(X;R)\cong R$.

A $CW$-complex $P$ is said to be an \emph{orientable Poincar\'e complex} 
if it satisfies Poincar\'e duality with respect to all choices of coefficient ring $R$ 
(see~\cite{Klein} for the \emph{non-orientable} definition). 
We say $P$ is \emph{finite} if it is finite as a $CW$-complex.
The \emph{dimension} of $P$ is the highest degree $n$ in which there is a
nonzero element in its $\mathbb{Z}$-cohomology, and we say $P$ is a Poincar\'e $n$-complex.

Any closed orientable $n$-manifold has the homotopy type of a finite Poincar\'e $n$-complex.
The classification of homotopy types of manifolds is then fittingly phrased in terms of classification of Poincar\'e complexes. 
The usual procedure is to first discard the local properties of manifolds, 
then use homotopy theoretic techniques to classify homotopy types of Poincar\'e complexes,
with local properties finally coming back into the picture when lifting the classification back to manifolds. 

Most work to date has involved the classification of low dimensional manifolds 
(see~\cite{Klein} for a more complete survey).  
That $1$-connected Poincar\'e $2$-complexes and Poincar\'e $3$-complexes 
have the homotopy type of a $2$-sphere and $3$-sphere respectively is an
easy consequence of Poincar\'e duality and the Hurewicz homomorphism.
Milnor~\cite{Milnor1} showed that the $\mathbb{Z}$-cohomology ring classifies $1$-connected Poincar\'e $4$-complexes,
while St{\"o}cker~\cite{Stocker} gave a list of four algebraic invariants that classify the homotopy types of 
$1$-connected orientable Poincar\'e $5$-complexes. 
Little is known beyond these dimensions. 
In the highly connected case, $(n-1)$-connected Poincar\'e $2n$-complexes
have been classified by Whitehead and Wall~\cite{Wall1}, while Sasao and Takahashi~\cite{SasTak} 
gave a partial solution for $(n-1)$-connected Poincar\'e $(2n+2)$-complexes.

The next step in the order of difficulty is to classify $(n-1)$-connected Poincar\'e $(2n+1)$-complexes. 
Though this is still generally an open problem, we can apply two homotopy theoretic simplifications that make this problem solvable.
First, we consider the classification problem after our spaces have been localized at some prime $p$. 
The motivation here is that localized spaces are much simpler from the perspective of homotopy theory,  
yet they retain much of the homotopy theoretic information of the original space
(never-the-less, it is sometimes possible to lift $p$-local results back to the category of integral spaces).
Second, we consider the classification problem after looping our spaces. 
Here one can often use the associative $H$-space structure on loop spaces 
to decompose them up to homotopy as a weak product of simpler spaces.
Spaces that are not homotopy equivalent sometimes have the same loop space homotopy decompositions,
so it is reasonable to expect that a loop space homotopy classification will be simpler. 
Since there is an adjoint isomorphism $\seqm{[\Sigma X,Y]}{\cong}{[X,\Omega Y]}$ of homotopy groups, 
a loop space homotopy classification is for many practical purposes as good as a homotopy classification of the original spaces.

Let the homology Bockstein operation
$$
\beta_r\colon \seqm{H_*(X;\zmodp)}{}{H_{*-1}(X;\zmodp)}
$$
be the composite 
\seqmm{H_*(X;\zmodp)}{\delta_r}{H_{*-1}(X;Z_{\zmod{r}})}{\rho_r}{H_{*-1}(X;\zmodp)},
where $\rho_r$ is the connecting map in the homology long exact sequence 
associated with the short exact sequence
$$
\seqmmmm{0}{}{\zmodp{r}}{}{\zmod{r+1}}{}{\zmodp}{}{0},
$$
and $\rho_r$ is induced by the reduction map \seqm{\zmodp{r}}{}{\zmodp}.
Then $\beta_r$ where $1\leq t\leq r$ detects \zmod{t} summands in the integral homology of $X$.
Taking duals one obtains the cohomology Bocksteins
$$
\beta_r\colon \seqm{H^*(X;\zmodp)}{}{H^{*+1}(X;\zmodp)}
$$
with similar properties.

We show that under a few assumptions, loop space homotopy types are uniquely determined by rational cohomology 
and the action of the Bocksteins operations $\beta_r$.

\begin{theorem}
\label{Main0}
Fix $m>2$, and let $M$ and $M'$ be $(2m-2)$-connected closed orientable $(4m-1)$-manifolds 
(or more generally finite orientable Poincar\'e complexes). 
Consider the following conditions:
\begin{itemize}
\item[(1)] $\beta_i(H^{2m-1}(M;\zmodp))\cong \beta_i(H^{2m-1}(M';\zmodp))$ for each $i>0$ and some prime $p$;
\item[(2)] $H^{2m-1}(M;\mathbb{Q})\cong H^{2m-1}(M';\mathbb{Q})$;
\item[(3)] $H^{2m-1}(M;\zmodtwo))=H^{2m-1}(M';\zmodtwo))=0$;
\item[(4)] $H^{2m-1}(M;\mathbb{Q})=H^{2m-1}(M';\mathbb{Q})=0$.
\end{itemize}
Then:
\begin{itemize}
\item[(i)] There is a homotopy equivalence localized at an odd prime $p$
$$
\Omega M_{(p)}\simeq\Omega M'_{(p)}
$$ 
if and only if conditions $(1)$ and $(2)$ hold. 
\item[(ii)] If conditions $(3)$ and $(4)$ hold, then there is a homotopy equivalence 
$$
\Omega M\simeq\Omega M'
$$ 
if and only if condition $(1)$ holds for all odd primes $p$. 
\end{itemize}~$\qqed$
\end{theorem} 

The proof is streamlined by assuming condition $(4)$, though with some extra work it can likely be dropped.

\section{mod-$p$ Poincar\'e Complexes}
\label{SModp}
If $X$ is a $1$-connected finite type $CW$-complex and $p$ a prime number, recall that 
the $p$-\emph{localization map} \seqm{X}{\ell}{X_{(p)}} induces a ring isomorphism
\seqm{H^*(X;\zmodp)}{\cong}{H^*(X_{(p)};\zmodp)}, and $X_{(p)}$ has a $p$-local $CW$-structure
(i.e. consisting of $p$-local cells, whose attaching maps are maps of $p$-localized spheres into $p$-local subcomplexes)
with $p$-local cells in one-to-one correspondance with generators of the $\zmodp$-module $H^*(X_{(p)};\zmodp)$.
Because an \emph{orientable Poincar\'e complex} $P$ satisfies Poincar\'e duality on mod-$p$ homology
for any prime $p$, then so does its $p$-localization $P_{(p)}$.  

Instead of working with the $p$-localization $P_{(p)}$ of a finite orientable Poincar\'e complex $P$,
it is more convenient to use the more general concept of a \emph{mod-$p$ Poincar\'e complex},
which is a finite $p$-local $CW$-complex  $Q$ that satisfies Poincar\'e Duality on its mod-$p$ cohomology. 
The \emph{dimension} $n$ is the highest degree in which there is a nonzero element in its mod-$p$ cohomology, 
and one says $Q$ is \emph{finite} if it has a finite number of $p$-local cells. 
The $(n-1)$-\emph{skeleton} of $Q$ in the $p$-local sense will be denoted by $\bar Q$.
Because the mod-$p$ cohomology generators of a $p$-local space are in one-to-one correspondance 
with its $p$-local cells, one can see that $Q$ is the cofiber of some map \seqm{S^{n-1}_{(p)}}{}{\bar Q}. 

Fix $n\geq 3$ and $k\geq 1$. We let $\mathcal{T}_{k,n}$ denote the set of classes of homotopy types of mod-$p$ 
Poincar\'e $(2n-1)$-complexes subject to the following conditions: the class $[W]$ is in $\mathcal{T}^{p}_{k,n}$ if and only if
\begin{itemize}
\item[(1)] $W$ is $(n-2)$-connected;
\item[(2)] $H^{n-1}(W;\zmodp)$ has rank $k$.
\end{itemize}
Fix some class $[W]\in\mathcal{T}^{p}_{k,n}$.
By mod-$p$ Poincar\'e duality and the first condition above, the 
$\zmodp$-submodule $H^{*}(\bar W;\zmodp)$ of $H^{*}(W;\zmodp)$ is described by an isomorphism
\begin{equation}
\label{EBasis} 
H^{*}(\bar W;\zmodp)\cong\zmodp\paren{x_{i}^*,y_{i}^*|1\leq i\leq k},
\end{equation}
where
$|x_{i}^*|=n-1$, $|y_{i}^*|=n$. 
We can and will choose the basis for $H^{*}(\bar W;\zmodp)$ to satisfy the following the conditions:
there is an integer $k_1$, with $0\leq k_1\leq k$, such that whenever $1\leq i\leq k_1$,
the action of the Bockstein operations on $H^{*}(\bar W;\zmodp)$ satisfy 
$$
\beta_{r_i}(x_{i}^*)=y_{i}^*
$$ 
for some choice of integer $r_{i}>0$ (depending on $i$),
and whenever $k_1<i\leq k$ we have
$$
\beta_{r}(x_{i}^*)=0
$$ 
for each $r>0$. 

Next, recall that for any group $G$, a Moore space $M(G,m)$ is unique up to homotopy type.
Thus $M(G_1\oplus G_2,m)\simeq M(G_1,m)\vee M(G_2,m)$. 
Therefore there exists a splitting of $\bar W$ as a wedge sum of 
$p$-localized Moore spaces $\M{n}{r_{i}}$ for $i\leq k_1$ corresponding to the action of $\beta_{r_i}$ on $x_i^*$, 
and $p$-localized spheres $S^{n-1}$ and $S^{n}$ corresponding to the generators 
$x_{i}^*$ and $y_{j}^*$ for $k_1<i\leq k$. 
In other words
\begin{equation}
\label{ESplit} 
\bar W\simeq_{(p)}\cvee{1\leq i\leq k_1}{}{\M{n}{r_i}}\vee\cvee{k_1< i\leq k}{}{(S^{n-1}\vee S^n)}.
\end{equation}

Fixing some generator $z^*\in H^{2n-1}(W;\zmodp)\cong\zmodp$, 
the cup product structure on $H^{*}(W;\zmodp)$ is described by a $k\times k$ \zmodp-matrix representation 
$$A_{z^*}=(a_{ij})$$
where $y_{j}^*x_{i}^*=a_{ij}z^*$.  
Set $k_2=k-k_1$. We partition $A_{z^*}$ into the block form 
\begin{equation}
\label{EMatrix}
A_{z^*}=\left(\begin{array}{cc}
B_{z^*} & D_{z^*} \\
C_{z^*} & E_{z^*}
\end{array}\right)
\end{equation} 
where 
$B_{z^*}$, $C_{z^*}$, $D_{z^*}$, $E_{z^*}$ 
are respectively matrices of dimensions
$k_1\times k_1$, $k_2\times k_1$, $k_1\times k_2$, $k_2\times k_2$. 
Since Poincar\'e Duality implies the cup product pairing as described above is nonsingular, 
the matrix $A_{z^*}$ is nonsingular.

\begin{proposition}
\label{P1}
Given a class $[W]\in\mathcal{T}^{p}_{k,n}$ and a generator $z^*\in H^{2n-1}(W;\zmodp)$ where $n\geq 3$, 
the nonsingular $A_{z^*}$ is such that $C_{z^*}$ is the zero $k_1\times k_2$ matrix,
and $B_{z^*}$ is symmetric when $n$ is even, and skew symmetric when $n$ is odd.  
\end{proposition}

See Section~\ref{SStruct} for a proof.
Since a $k\times k$ skew symmetric matrix is nonsingular if and only if $k$ even, the following is immediate:  
\begin{corollary}
There exist no classes $[W]\in\mathcal{T}^{p}_{2l+1,2m+1}$ such that $H^{2m-1}(W;\mathbb{Q})=0$ or $H^{2m}(W;\mathbb{Q})=0$. 
\end{corollary}

We therefore focus on those classes $[W]\in\mathcal{T}^{p}_{k,n}$ for $n$ even. 
The above statements are proved in the next two sections. 
We see that the homotopy type of $\Omega W$ is completely classified 
by rational cohomology and the action of the Bocksteins operations as follows:
\begin{theorem}
\label{Main}
Let $[W],[W']\in\mathcal{T}^{p}_{k,2m}$ and $m>2$. 
Then $\Omega W\simeq\Omega W'$ if and only if the following are satisfied: 
\begin{itemize}
\item[(1)] $H^{2m-1}(W;\mathbb{Q})\cong H^{2m-1}(W';\mathbb{Q})$;
\item[(2)] $\beta_i(H^{2m-1}(W;\zmodp))\cong \beta_i(H^{2m-1}(W';\zmodp))$ for each $i>0$.
\end{itemize}~$\qqed$
\end{theorem}

The subscript $(p)$ is suppressed until the last section of this paper, 
and we assume that all our spaces are $p$-local, or localized at $p$ where appropriate, for some fixed odd prime $p$.
Any reference to a $CW$-structure on a $p$-local space is always taken to be in the $p$-local sense.

\section{Mod-$p$ Loop Space Homology}
\label{SLHom}

Keep $p$ fixed as an odd prime number.
Let $\bar P$ be a finite type $CW$-complex, $P$ be the cofiber of some map
$$
\alpha\colon\seqm{S^{n-1}}{}{\bar P}
$$ 
for some fixed integer $n>3$, and
$$
i\colon\seqm{\bar P}{}{P}
$$
denote the inclusion. Let 
$$
\alpha'\colon\seqm{S^{n-2}}{}{\Omega\bar P}
$$
be the adjoint of $\alpha$. 
Since $i\circ\alpha'$ is null homotopic, the algebra map
$$
(\Omega i)_*\colon\seqm{H_*(\Omega\bar P;R)}{}{H_*(\Omega P;R)}
$$
factors through a map
\begin{equation}
\label{ECellAttach}
\theta\colon\seqm{H_*(\Omega \bar P;R)/I}{}{H_*(\Omega P;R)},
\end{equation}    
for any ring $R$, where $I$ is the two-sided ideal generated by the image of of $\alpha'_*$
in degree $n-2$.

Determining the conditions under which $\theta$ is a Hopf Algebra isomorphism is called
the \emph{cell attachment problem}. 
This has a long history, having been studied by Anick~\cite{Anick1}, Bubenik~\cite{Bubenik}, F\'elix and Thomas~\cite{FelixThomas}, 
and Halperin, Hess, and Lemaire~\cite{Lemaire,HalperinLemaire,HalperinLemaire2,HessLemaire}. 
Lemaire~\cite{Lemaire}, for one, found that $\theta$ is a Hopf algebra isomorphism 
whenever the morphism of graded $R$-vector spaces
$$
Tor^{\pi}_{p}\colon\seqm{Tor^{H_*(\Omega\bar P;R)}_{p}(R,R)}{}{Tor^{H_*(\Omega\bar P;R)/I}_{p}(R,R)}
$$
induced by the canonical surjection \seqm{H_*(\Omega X;R)}{\pi}{H_*(\Omega X;R)/I} is bijective,
and $R$ is a vector space of characteristic $p$. 
 
Our goal in this section is to determine conditions on the cohomology ring $H_*(\Omega P;\zmodp)$ which make $\theta$ a Hopf algebra isomorphism.  
Many highly connected mod-$p$ Poincar\'e complexes, including the ones dealt with in this paper, are covered here.
The ideal $I$ is computed in Proposition~\ref{L1}, 
and in Theorem~\ref{P0} we use a Leray-Serre spectral sequence approach to arrive at the Hopf algebra isomorphism $\theta$.

Fix any integer $m>1$ such that $n\geq m$. 
Assume our $CW$-complex $P$ is $(m-1)$-connected $n$-dimensional, 
with mod-$p$ reduced homology generated by $a_1,\ldots,a_\ell$ and $z$, where 
$$
m=|a_1|\leq |a_2|\leq\cdots\leq |a_\ell|<|z|=n.
$$
Whenever $|a_i|+|a_j|=n$, 
let the integer $c_{ij}$ be such that $a_{j}^*a_{i}^*=c_{ij}z^*$,
where $a_{i}^*$, $a_{j}^*$, $z^*$ are the cohomology duals of $a_i$, $a_{j}$ and $z$,
and in addition, we make the following assumptions: 
\begin{itemize}
\item[(1)] $\dim\bar P<\dim P$, meaning $\bar P$ is the $(n-1)$-skeleton of $P$; 
\item[(2)] $3(m-1)>n-2$ and $n$ is odd; 
\item[(3)] $\bar P\simeq \Sigma^2 X$ for some $X$.
\end{itemize}

As $\bar P$ is homotopy equivalent to a suspension,
cup products are trivial on $\bar H^*(\bar P;\zmodp)\subseteq \bar H^*(P;\zmodp)$,
implying the $c_{ij}$'s describe the cup-product structure for
$\bar H^*(P;\zmodp)$.

Consider the mod-$p$ homology Serre spectral sequences $\bar E$ and $E$ for the path fibrations of $\bar P$ and $P$,
and the morphism of spectral sequences
$$
\gamma\colon\seqm{\bar E}{}{E}
$$
induced by the inclusion \seqm{\bar P}{i}{P}.
Here we have
$$
\bar E^2_{*,*}=H_*(\bar P;\zmodp)\otimes H_*(\Omega \bar P;\zmodp),
$$
$$
E^2_{*,*}=H_*(P;\zmodp)\otimes H_*(\Omega P;\zmodp).
$$
The corresponding mod-$p$ cohomology spectral sequences 
are labelled by switching subscripts with superscripts as is standard.

Since $\bar P$ is homotopy equivalent to a suspension $\Sigma^2 X$, 
the basis elements $a_i$ of $H_*(\bar P;\zmodp)$ and $H_*(P;\zmodp)$ are transgressive.
Thus, let 
$$u_i=\tau(a_i)\in H_*(\Omega P;\zmodp),$$
$$\bar u_i=\tau(a_i)\in H_*(\Omega \bar P;\zmodp)$$
be the transgressions of the $a_i$'s.
Note that there is a Hopf algebra isomorphism
$$
H_*(\Omega\bar P;\zmodp)\cong T(\bar H_*(\Sigma X;\zmodp))\cong T(\bar u_1,\ldots,\bar u_\ell).
$$
Since $\Sigma X$ is a suspension, cup products on $\bar H_*(\Sigma X;\zmodp))$ are trivial, 
implying the algebra generators $\bar u_i$ are primitive.
The algebra map $(\Omega i)_*$ satisfies 
$(\Omega i)_*(\bar u_\ell)=u_\ell$.
Then on the second page of spectral sequences, 
$$\gamma^2(1\otimes\bar u_i)=1\otimes u_i$$
$$\gamma^2(a_i\otimes 1)=a_i\otimes 1.$$ 

Fix $m'$ to be the smallest integer such that there is a $c_{ij}$ prime to $p$ 
for some $i$ and $j$ satisfying $i\leq j$ and $|a_i|=m'$. 
If no such integer exists, set $m'=|z|=n$.
We now state some properties of the spectral sequences $\bar E$ and $E$.

\begin{proposition}
\label{L1a}
The following hold:
\begin{romanlist}
\item The kernel of the map 
$$
(\Omega i)_*\colon\seqm{H_{n-2}(\Omega\bar P;\zmodp)}{}{H_{n-2}(\Omega P;\zmodp)}
$$
is generated by $\alpha'_*(\iota_{n-2})$, 
where $\iota_{n-2}\in H_*(S^{n-2};\zmodp)\cong\zmodp$ is a generator. 

When $m'=n$, we have $\alpha'_*(\iota_{n-2})=0$.
\item $\bar d^{r}(\bar E^{r}_{i,j})=\{0\}$ for $2\leq r<i$, 
and $d^{r}(E^{r}_{i,j})=\{0\}$ for $2\leq r<i$ and $i\neq n$.  
\item $d^r(z\otimes 1)=0$ for $r<m'$, so $z\otimes 1$ survives to $E^{m'}_{n,0}$. 
The differential 
$$d^{m'}\colon \seqm{E^{m'}_{n,0}}{}{E^{m'}_{n-m',m'-1}}$$ 
satisfies
\begin{align*}
d^{m'}(z\otimes 1)=&
\begin{cases}
(-1)^{m'}\csum{|a_i|=m',|a_j|=n-m'}{}{c_{ij} (a_j\otimes u_i)},&\mbox{if }m'<n;\\
1\otimes \tau(z)\neq 0,&\mbox{if }m'=n.
\end{cases}
\end{align*}
\item The map
$$
\gamma^{r}\colon\seqm{\bar E^{r}_{i,j}}{}{E^{r}_{i,j}}
$$
is an isomorphism for
$2\leq r\leq i$, $j<n-2$, $i\neq n-m'$, and $i\neq n$.
It is also an isomorphism for $2\leq r\leq i$, $i=n-m'$ and $j<m'-1$. 
When $i=0$, it is an isomorphism for $j<n-2$ and all $r\geq 2$.
\item When $m'<n$, the map
$$
\gamma^{r}\colon\seqm{\bar E^{r}_{n-m',m'-1}}{}{E^{r}_{n-m',m'-1}}
$$
is an isomorphism for $r>n-m'$ and $2\leq r\leq m'$, 
and the projection of the element
$$
\zeta=(-1)^{m'}\csum{|a_i|=m',|a_j|=n-m'}{}{c_{ij} (a_j\otimes \bar u_i)}
$$
generates its kernel for $(m'+1)\leq r\leq n-m'$. 
\item The projection of $1\otimes(\alpha')_*(\iota_{n-2})$ to $\bar E^{r}_{0,n-2}$ generates the kernel of
$$
\gamma^r\colon\seqm{\bar E^{r}_{0,n-2}}{}{E^{r}_{0,n-2}}
$$
for $2\leq r\leq n-m'$.
\end{romanlist}
\end{proposition}

\begin{proof}[Proof of part $(i)$]

Notice there is the following homotopy commmutative diagram
\begin{equation}
\label{DFib}
\diagram
S^{n-1}\rto^{\alpha}\dto^{\ell} & \bar P \ddouble\rto^{i} & P\ddouble\\
F\rto^{f} & \bar P\rto^{i} & P,
\enddiagram
\end{equation}
where the top row is the cofibration sequence for the map $\alpha$,
$F$ is the homotopy fiber of the inclusion \seqm{\bar P}{i}{P},
the bottom row corresponding homotopy fibration sequence,
and $\ell$ is some lift. Since \seqm{\bar P}{i}{P} induces an
isomorphism on mod-$p$ homology in degrees less than $n$,
$F$ is at least $(n-2)$-connected. 
It is well known that fibers and cofibers agree in the stable range. 
That is, the lift $\ell$ induces an isomorphism on mod-$p$ homology in degrees less than $m+n-1$. 
Thus $\ell$ is an inclusion into the bottom sphere 
inducing in isomorphism in degree $n-1$ mod-$p$ homology, 
and the adjoint 
\seqm{S^{n-2}}{\ell'}{\Omega F} 
of $\ell$ induces an isomorphism in degree $n-2$.

By the mod-$p$ homology Serre exact sequence for the homotopy fibration
$$
\seqmm{\Omega F}{\Omega f}{\Omega\bar P}{\Omega i}{\Omega P}
$$
the image of $(\Omega f)_*$ is equal to the kernel of $(\Omega i)_*$ in degree $n-2$. 
By the left homotopy commutative square in diagram~(\ref{DFib}), 
$\alpha'$ is homotopic to
$$
\seqmm{S^{n-2}}{\ell'}{\Omega F}{\Omega f}{\Omega\bar P}.
$$ 
Since $\ell'$ induces an isomorphism in degree $n-2$,
the element $\alpha'_*(\iota_{n-2})$ must generate the kernel of
$(\Omega i)_*$ in degree $n-2$.

To see that $\alpha'_*(\iota_{n-2})=0$ whenever $m'=n$,
by part $(iv)$ the map 
$$
\gamma^{r}\colon\seqm{\bar E^{r}_{r,n-r-1}}{}{E^{r}_{r,n-r-1}}
$$
is an isomorphism for $r\geq 2$. 
Since $\bar E^{\infty}=E^{\infty}=\{0\}$, 
$\bar E^{r+1}_{0,n-2}=\frac{\bar E^{r}_{0,n-2}}{\bar d^r(\bar E^r_{r,n-r-1})}$,
and
$E^{r+1}_{0,n-2}=\frac{E^{r}_{0,n-2}}{d^r(E^r_{r,n-r-1})}$,
it follows that 
$$
\gamma^{r}\colon\seqm{\bar E^{r}_{0,n-2}}{}{E^{r}_{0,n-2}}
$$
is an isomorphism for all $r\geq 2$.
Since $\alpha'_*(\iota_{n-2})$ generates the kernel of $(\Omega i)_*$ in degree $n-2$,
$1\otimes\alpha'_*(\iota_{n-2})\in\bar E^{2}_{0,n-2}$ generates the kernel of $\gamma^2$.
But $\gamma^2$ being an isomorphism implies $1\otimes\alpha'_*(\iota_{n-2})=0$,
so $\alpha'_*(\iota_{n-2})=0$.

\end{proof}

\begin{proof}[Proof of part $(ii)$]
Since path fibrations are principal fibrations, 
differentials for the spectral sequence of $\bar E$ 
for the path fibration of $\bar P$ satisfy 
$$
\bar d^r(a\otimes b)=(1\otimes b)\bar d^r(a\otimes 1)
$$ 
for any $a,b$. 
Similarly for the spectral sequence $E$ for the path fibration of $P$.

Since the elements $a_i\otimes 1\in \bar E^2_{|a_i|,0}$ are transgressive,
$$
\bar d^r(a_i\otimes 1)=0
$$ 
for $r<|a_i|$. 
Since $\gamma^2(a_i\otimes 1)=a_i\otimes 1\in E^2_{|a_i|,0}$,
the result follows by naturality of spectral sequences.

\end{proof}

\begin{proof}[Proof of part $(iii)$]

In $E^{|a_i|}_{*,*}$ we have the differentials
$$
d_{|a_i|}(1\otimes u_i^*)=a_i^*\otimes 1, d^{|a_i|}(a_i\otimes 1)=1\otimes u_i,
$$
where $u_i^*$ and $a_i^*$ are the mod-$p$ cohomology duals. 
Since we are assuming $|a_1|\leq\cdots\leq |a_\ell|$,
we have $|a_1|=\min\{|a_1|,\ldots,|a_\ell|\}$. 
When $|a_i|+|a_j|=n$, 
$$
\begin{array}{rcl}
d_{m'}(a_j^*\otimes u_i^*)&=&(-1)^{m'}(a_j^*\otimes 1)d_{m'}(1\otimes u_i^*)\\
&=&(-1)^{m'}(a_j^* a_i^*)\otimes 1\\ 
&=&(-1)^{m'}c_{ij}z^*\otimes 1,
\end{array}
$$
and since $c_{ij}$ is divisible by $p$ whenever $|a_j|+|a_i|=n$ and $|a_i|<m'$,  
it follows that $d_{|a_j|}(a_j^*\otimes u_i^*)=0$.
Thus, for $r<m'$, the differentials 
$$
d_{r}\colon\seqm{E_{r}^{n-r,r-1}}{}{E_{r}^{n,0}}
$$ 
in the cohomology spectral sequence for the path fibration of $P$ are zero.
By duality of the spectral sequence, $d^{r}(z\otimes 1)=0$ whenever $r<m'$,  
and so we can project $z\otimes 1$ to $E^{m'}_{n,0}$.

When $m'=|z|=n$, we see that $z\otimes 1$ is transgressive, with 
$$
d^{n}(z\otimes 1)=w
$$ 
for some nonzero $w\in E^{n}_{0,n-1}$.
On the other hand, when $m'<n$, we have some integers $c'_{ij}$ such that 
$$
d^{m'}(z\otimes 1)=\csum{|a_i|=m',|a_j|=n-m'}{}{c'_{ij}(a_j\otimes u_i)}
$$
From the duality of the spectral sequence, 
$$
\begin{array}{rcl}
(-1)^{m'}c_{ij}=\vbracket{(-1)^{m'}c_{ij}z^*, z}&=&\vbracket{d_{m'}(a_j^*\otimes u_i^*), z}\\
&=&\vbracket{a_j^*\otimes u_i^*, d^{m'}(z\otimes 1)}\\
&=&\vbracket{a_j^*\otimes u_i^*,\csum{|a_s|=m',|a_t|=n-m'}{}{c'_{st} a_s\otimes u_t}}=c'_{ij}.
\end{array}
$$

\end{proof}

\begin{proof}[Proof of part $(iv)$]
Since the inclusion \seqm{\bar P}{i}{P} induces isomorphisms
$\seqm{H_i(\bar P)}{i_*}{H_i(P)}$ for $i\neq n$,
by the Serre homology exact sequences for the path fibrations of $\bar P$ and $P$, 
the map 
\seqm{\Omega\bar P}{\Omega i}{\Omega P} 
induces isomorphisms
\seqm{H_i(\Omega\bar P)}{(\Omega i)_*}{H_i(\Omega P)} for $i<n-2$.
Therefore 
\seqm{\bar E^2_{i,j}}{\gamma^2}{E^2_{i,j}} 
is an isomorphism for $i\neq n$ and $j<n-2$.

By parts $(ii)$ and $(iii)$,
elements in $E^{r}_{i,j}$ are in the image of a differential $d^r$
only when $i=0$, or when $i=n-m'$, $r=m'$, and $j\geq m'-1$. 
By part $(ii)$ differentials $d^r$ are zero on $E^{r}_{i,*}$ 
when $i\neq n$ and $2\leq r<i$.
The above also holds true for the spectral sequence 
$\bar E$ in place of $E$.
Therefore $\gamma^2$ extends to isomorphisms
\seqm{\bar E^r_{i,j}}{\gamma^r}{E^r_{i,j}} 
for $2\leq r\leq i$, $j<n-2$, $i\neq n-m'$, and $i\neq n$,
or when $2\leq r\leq i$, $i=n-m'$ and $j<m'-1$.

For the case $i=0$, suppose by induction 
\seqm{\bar E^r_{0,j}}{\gamma^r}{E^r_{0,j}}
is an isomorphism for some $r\geq 2$ and $j<n-2$.
This is true when $r=2$. As we saw above
\seqm{\bar E^r_{r,j-r+1}}{\gamma^r}{E^r_{r,j-r+1}}
is an isomorphism when $r\neq n-m'$, $r\neq n$, and $r\geq 2$.
Therefore it is an isomorphism for all choices of $(r,j-r+1)$ 
lying on the positive quadrant with $j<n-2$
(since $j<n-2$ implies $j-r+1<m'-1$ whenever $r=n-m'$, 
and $r<n$ whenever $j-r+1=0$). 
Since
$\bar E^{r+1}_{0,j}=\frac{\bar E^{r}_{0,j}}{\bar d^r(\bar E^r_{r,j-r+1})}$,
and
$E^{r+1}_{0,j}=\frac{E^{r}_{0,j}}{d^r(E^r_{r,j-r+1})}$,
it follows that 
\seqm{\bar E^{r+1}_{0,j}}{\gamma^{r+1}}{E^{r+1}_{0,j}}
is also an isomorphism, and the induction is finished.

\end{proof}

\begin{proof}[Proof of part $(v)$]

From the proof of part $(iv)$,
\seqm{\bar E^2_{n-m',m'-1}}{\gamma^2}{E^2_{n-m',m'-1}}
is an isomorphism. That $\gamma^2$ extends to isomorphisms
\seqm{\bar E^r_{n-m',m'-1}}{\gamma^r}{E^r_{n-m',m'-1}}
for $2\leq r\leq m'$ follows by parts $(ii)$ and $(iii)$.

Since $E^{m'+1}_{n-m',m'-1}=\frac{E^{m'}_{n-m',m'-1}}{d^{m'}(z\otimes 1)}$,
$\zeta$ generates the kernel of
\seqm{\bar E^{m'+1}_{n-m',m'-1}}{\gamma^{m'}}{E^{m'+1}_{n-m',m'-1}}.

We have 
$\bar E^{r}_{n-m',m'-1}\cong\bar E^{r}_{n-m',m'-1}$
and 
$E^{r}_{n-m',m'-1}\cong E^{r}_{n-m',m'-1}$
for $m'+1\leq r\leq n-m'$ following from part $(ii)$.
Hence the projection of $\zeta$ generates the kernel of
\seqm{\bar E^{r}_{n-m',m'-1}}{\gamma^{r}}{E^{r}_{n-m',m'-1}}
for $m'+1\leq r\leq n-m'$.

\end{proof}

\begin{proof}[Proof of part $(vi)$]

By part $(i)$, $(\alpha')_*(\iota_{n-2})$ generates the kernel of 
$(\Omega i)_{*}$,
so $1\otimes(\alpha')_*(\iota_{n-2})$ generates the kernel of
\seqm{\bar E^{2}_{0,n-2}}{\gamma^2}{E^{2}_{0,n-2}}.

Fix some $2\leq r<n-m'$, and suppose the projection of 
$1\otimes(\alpha')_*(\iota_{n-2})$
generates the kernel of 
\seqm{\bar E^{r}_{0,n-2}}{\gamma^{r+1}}{E^{r}_{0,n-2}}.

Recall from part $(iv)$ that  
\seqm{\bar E^{r}_{r,n-r-1}}{\gamma^r}{E^{r}_{r,n-r-1}}
is an isomorphism for $2\leq r<n-m'$. 
Then because
$\bar E^{r+1}_{0,n-2}=\frac{\bar E^{r}_{0,n-2}}{\bar d^r(\bar E^r_{r,n-r-1})}$
and
$E^{r+1}_{0,n-2}=\frac{E^{r}_{0,n-2}}{d^r(E^r_{r,n-r-1})}$,
the projection of $1\otimes(\alpha')_*(\iota_{n-2})$ 
also generates the kernel of 
\seqm{\bar E^{r+1}_{0,n-2}}{\gamma^{r+1}}{E^{r+1}_{0,n-2}},
and the result follows by induction.

\end{proof}

For elements $x$ and $y$ in a graded algebra, let $[x,y]=xy-(-1)^{|x||y|}yx$ denote the graded Lie bracket.
In the next lemma we use the following sets for indexing:
$$
\mathcal{A}_{s,k}=\{(i,j)|k<i<j\leq \ell\mbox{, }|a_i|=s\mbox{, }|a_j|=n-s\},
$$
$$
\mathcal{B}_k=\cunion{s}{}{\mathcal{A}_{s,k}}=\{(i,j)|k<i<j\leq \ell\mbox{, }|a_i|+|a_j|=n\}.
$$

\begin{proposition}
\label{L1}
Set $\eta =\ceil{\frac{n}{2}}$. 
Consider the following elements in $H_{n-2}(\Omega\bar P;\zmodp)$ for $m\leq s\leq\eta$:
$$
\kappa_s=\csum{(i,j)\in\mathcal{A}_{s,0}}{}{c_{ij}[\bar u_i,\bar u_j]},
$$
Let $m'\geq m$ be the smallest integer such that there is a $c_{ij}$ prime to $p$ for some $i\leq j$,
with $i$ satisfying $|a_i|=m'$. If no such integer exists, set $m'=|z|=n$.

Then there exist integers $b_m,b_{m+1},\ldots,b_{\eta}$, each prime to $p$, 
such that 
$$
\alpha'_*(\iota_{n-2})=\csum{s=m}{\eta}{(-1)^{s}b_s\kappa_s}.
$$
\end{proposition}

We will prove Proposition~\ref{L1} 
using an inductive argument on the skeleta of $P$.
We describe the setup for this before going into the proof. 

Since $|a_1|\leq\cdots\leq |a_\ell|$,
one can take the subcomplex $Y_k$ of $\bar P$,
with $H_*(Y_k;\zmodp)$ generated by $a_1,\ldots,a_k$.
Here $Y_\ell=\bar P$, and $Y_0$ is a point.
Let us consider the quotients $P/Y_k$ and $\bar P/Y_k$.
Note $P/Y_0=P$, $\bar P/Y_0=\bar P$, $P/Y_\ell=S^n$,
and $\bar P/Y_\ell=\ast$.
Abusing notation, $P/Y_k$ has reduced mod-$p$ homology generated by 
$a_{k+1},\ldots,a_\ell$, and the single degree $n$ generator $z$.
The non-trivial cup products on $\bar H^*(P/Y_k;\zmodp)$
are described by $a_{j}^*a_{i}^*=c_{ij}z^*$ whenever $|a_{j}|+|a_{i}|=n$,
and $j\geq k$. 
We fix $m_k$ to be the smallest integer so that there is a $c_{ij}$ prime to $p$ 
for some $i$ and $j$ satisfying $k\leq i\leq j$ and $|a_i|=m_k$. 
If no such integer exists, set $m_k=|z|=n$.

The $(n-1)$-skeleton $P/Y_k$ is $\bar P/Y_k$. 
Let 
$\alpha_k\colon\seqm{S^{n-1}}{}{\bar P/Y_k}$ 
be the attaching map for the single $n$-cell of $P/Y_k$, 
$$
\alpha'_k\colon\seqm{S^{n-2}}{}{\Omega\bar P/Y_k}
$$ 
be the adjoint of $\alpha_k$, and
$$
i_k\colon\seqm{\bar P/Y_k}{}{P/Y_k}
$$
the inclusion of the $(n-1)$-skeleton. 

Let $(\bar E_k)$ and $(E_k)$ be the mod-$p$ homology spectral sequence
for the path-space fibration of $\bar P/Y_k$ and $P/Y_k$ respectively,
and 
$$
\gamma\colon\seqm{(\bar E_{k})}{}{(E_{k})}
$$
be the morphism of spectral sequences induced by $i_k$.

Since $\bar P$ is a double suspension, so is $\bar P/Y_k$. 
Just as before, we have generators 
$u_{k+1},\ldots,u_\ell\in H_{*}(\Omega P/Y_{k-1};\zmodp)$, 
and $\bar u_{k+1},\ldots,\bar u_\ell \in H_{*}(\Omega(\bar P/Y_{k-1});\zmodp)$,
that are the trangressives of $a_{k+1},\ldots,a_\ell$,
with the $\bar u_i$'s being primitive.

\begin{remark}
\label{rReplace}
The spaces $P/Y_k$ satisfy the same basic properties as $P$ outlined at the beginning of the section.
Then Proposition~\ref{L1a} applies for $P/Y_k$ in place of $P$.

More precisely, in Proposition~\ref{L1a} we can take 
$P/Y_k$, $\bar P/Y_k$, $\alpha_k$, $i_k$, $(\bar E_k)$, $(E_k)$, and $m_k$
in place of $P$, $\bar P$, $\alpha$, $i$, $\bar E$, $E$, and $m'$ respectively.
The sums in parts $(ii)$ and $(iii)$ of Proposition~\ref{L1a} are taken with respect to 
the basis elements $a_{k+1},\ldots,a_\ell$ of $H_*(P/Y_k;\zmodp)$.

\end{remark}

\begin{proof}[Proof of Proposition~\ref{L1}]

The proof proceeds using induction.
At each stage we show that the proposition holds for each quotient $P/Y_k$ in place of $P$.
The induction starts with the base case $P/Y_\ell=S^n$ and ends with $P/Y_0=P$. 

Assume Proposition~\ref{L1} holds for the quotient $P/Y_{k}$, for some $1\leq k\leq \ell$. 
That is, let us assume
$(\alpha'_k)_*(\iota_{n-2})=\chi_k$, 
where we set
$$
\chi_k=\csum{s=|a_{k+1}|}{\eta}{(-1)^{s}b_s\kappa_{s,k}},
$$
and
$$
\kappa_{s,k}=\csum{(i,j)\in\mathcal{A}_{s,k}}{}{c_{ij}[\bar u_i,\bar u_j]}
$$
in $H_*(\Omega(\bar P/Y_k);\zmodp)$.
The base case $k=\ell$ and $P/Y_\ell=S^n$ is clearly true.

Since $\bar P/Y_{k-1}$ is a double suspension, 
the elements $\bar u_k,\ldots,\bar u_\ell$ in 
$$
H_{*}(\Omega(\bar P/Y_{k-1});\zmodp)\cong T(\bar u_k,\ldots,\bar u_\ell)
$$
are primitive. 
Since $3(m-1)>n-2$, $H_{n-2}(\Omega(\bar P/Y_{k-1});\zmodp)$ has no monomials of length greater than $2$,
and so the brackets $[\bar u_i,\bar u_j]$ subject to $(i,j)\in \mathcal{B}_{k-1}$
form a basis for the primitives in $H_{n-2}(\Omega(\bar P/Y_k);\zmodp)$ 
(note that $n$ is odd implies $i\neq j$). 
Because $\iota_{n-2}$ is primitive, 
$(\alpha'_{k-1})_*(\iota_{n-2})$ is a primitive element in $H_{n-2}(\Omega(\bar P/Y_{k-1});\zmodp)$, 
and so for some integers $c''_{ij}$ we can set
$$
(\alpha'_{k-1})_*(\iota_{n-2})=\csum{(i,j)\in\mathcal{B}_{k-1}}{}{c''_{ij}[\bar u_i,\bar u_j]}.
$$

Take the quotient map
$$
q_{k-1}\colon\seqm{\bar P/Y_{k-1}}{}{\bar P/Y_k}.
$$
Notice that $\alpha_k$ factors as $q_{k-1}\circ\alpha_{k-1}$, 
so $\alpha'_k$ factors as 
$$
\alpha'_k\colon\seqmm{S^{n-2}}{\alpha'_{k-1}}{\Omega(\bar P/Y_{k-1})}{\Omega q_{k-1}}{\Omega(\bar P/Y_k)}.
$$
Since the algebra map $(\Omega q_{k-1})_*$ sends $\bar u_i$ to $\bar u_i$ for $i>k$, 
and $\bar u_k$ to $0$, in $H_{n-2}(\Omega(\bar P/Y_k);\zmodp)$ we have
$$
(\alpha'_k)_*(\iota_{n-2})=\csum{(i,j)\in\mathcal{B}_k}{}{c''_{ij}[\bar u_i,\bar u_j]}.
$$
But $(\alpha'_k)_*(\iota_{n-2})=\chi_k$ by our inductive assumption,  
so by comparing coefficients
$$
c''_{ij}=(-1)^{|a_i|}b_{|a_i|}c_{ij}
$$
whenever $(i,j)\in\mathcal{B}_k$.
That is, whenever $k<i<j\leq \ell$ and $|a_i|+|a_j|=n$.
Therefore, in order to show
\begin{equation}
\label{EImage}
(\alpha'_{k-1})_*(\iota_{n-2})=\chi_{k-1}=\csum{s=|a_{k}|}{\eta}{(-1)^{s}b_s\kappa_{s,k-1}},
\end{equation}
where 
$$
\kappa_{s,k-1}=\csum{(i,j)\in\mathcal{A}_{s,k-1}}{}{c_{ij}[\bar u_i,\bar u_j]}, 
$$
we note that $\mathcal{A}_{|a_k|,k}\subseteq\mathcal{A}_{|a_k|,k-1}$ 
and $\mathcal{A}_{s,k}=\mathcal{A}_{s,k-1}$ when $s>|a_k|$,
and so we are left to show there exists an integer $b_{|a_k|}$ prime to $p$
such that $c''_{ij}=(-1)^{|a_k|}b_{|a_k|}c_{ij}$, for $i$ and $j$ satisfying $k\leq i<j$, 
$|a_i|=|a_k|$, and $|a_i|+|a_j|=n$.

Using Remark~\ref{rReplace} and part $(iii)$ of Proposition~\ref{L1a},
\begin{align}
\label{EDiff}
d^{m_k}(z\otimes 1)=&
\begin{cases}
(-1)^{m_k}\csum{|a_i|=m_k,|a_j|=n-m_k}{}{c_{ij} (a_j\otimes u_i)},&\mbox{if }m_k<n\\
1\otimes \tau(z),&\mbox{if }m_k=|z|=n
\end{cases}
\end{align}
for some nonzero $1\otimes \tau(z)\in (E_{k-1})^{m_k}_{0,n-1}$.

By part $(v)$ of Proposition~\ref{L1a}, the map
\begin{equation}
%\label{EKer1}
\gamma^{r}\colon\seqm{(\bar E_{k-1})^{r}_{n-m_k,m_k-1}}{}{(E_{k-1})^{r}_{n-m_k,m_k-1}}
\end{equation}
is an isomorphism for $2\leq r\leq m_k$,
and the following element
\begin{equation}
\label{EZeta}
\zeta=(-1)^{m_k}\csum{|a_i|=m_k,|a_j|=n-m_k}{}{c_{ij} (a_j\otimes \bar u_i)}
\end{equation}
in $(\bar E_{k-1})^{r}_{n-m_k,m_k-1}$,
with the sum taken with respect to the basis $a_k,\ldots,a_\ell$ of $H_*(\bar P/Y_{k-1};\zmodp)$,
generates the kernel of $\gamma^{r}$ for $m_{k}+1\leq r\leq n-m_k$
(note: $\bar d^r(\bar E_{k-1})^{r}_{n-m_k,m_k-1}=\{0\}$ for $r$ in this range, so we can project $\zeta$).

%Note we have a cofibration sequence
%$$
%\seqmm{S^{n-1}}{\alpha_{k-1}}{\bar P/Y_{k-1}}{i_{k-1}}{P/Y_{k-1}}.
%$$

By part $(i)$ of Proposition~\ref{L1a},
$(\alpha'_{k-1})_*(\iota_{n-2})$ generates the kernel of 
$$
(\Omega i_{k-1})_{*}\colon\seqm{H_{n-2}(\Omega (\bar P/Y_{k-1});\zmodp)}{\Omega i_{k-1}}{H_{n-2}(\Omega (P/Y_{k-1});\zmodp)},
$$
and by part $(vi)$, the projection of $1\otimes (\alpha'_{k-1})_*(\iota_{n-2})$ 
generates the kernel of
\begin{equation}
\label{EKer2}
\gamma^r\colon\seqm{(\bar E_{k-1})^{r}_{0,n-2}}{}{(E_{k-1})^{r}_{0,n-2}}
\end{equation}
for $2\leq r\leq n-m_k$. In particular 
$$
\gamma^r(1\otimes (\alpha'_{k-1})_*(\iota_{n-2}))=0
$$
for $r\geq 2$.

We now return to showing there exists a $b_{|a_k|}$ prime to $p$
such that $c''_{kj}=(-1)^{|a_k|}b_{|a_k|}c_{ij}$ for $i$ and $j$ such that $k\leq i< j$, 
$|a_i|=|a_k|$, and $|a_i|+|a_j|=n$. 
The case $m_k=|z|=n$ is easy. Here we have $c_{ij}=0$ for each choice of $i,j$,
and $(\alpha'_{k-1})_*(\iota_{n-2})=0$ by Remark~\ref{rReplace} 
and part $(i)$ of Proposition~\ref{L1a}. 
Then $c''_{kj}=0$, and we can set $b_{|a_k|}=1$. 
Let us therefore focus on the case $m_k<n$. 
Fix $q=|a_k|$. By definition $m_k\geq q$,
and since $m_k<n$ and $n$ is odd, $q<n-q$. 

Let us again recall that path fibrations are principal fibrations,
and so our differentials satisfy 
$d^r(a\otimes b)=(1\otimes b)d^r(a\otimes 1)$
and 
$\bar d^r(a\otimes b)=(1\otimes b)\bar d^r(a\otimes 1)$.

Consider the following element 
\begin{equation}
\label{EZeta1}
\zeta''=\csum{|a_i|=q,|a_j|=n-q}{}{c''_{ij} (a_j\otimes \bar u_i)}
\end{equation}
in $(\bar E_{k-1})^{r}_{n-q,q-1}$ for $2\leq r\leq n-q$.
Since 
$d^{|a_i|}(a_i\otimes\bar u_j)=1\otimes\bar u_j \bar u_i$ in $(\bar E_{k-1})^{|a_i|}_{0,n-2}$,
then $1\otimes\bar u_j \bar u_i=0$ in $(\bar E_{k-1})^{r}_{0,n-2}$ when $r>|a_i|$,
and so $1\otimes[\bar u_i,\bar u_j]=1\otimes\bar u_i \bar u_j$ in $(\bar E_{k-1})^{r}_{0,n-2}$ 
under the condition that $|a_i|<|a_j|$.
Because $q=|a_k|=\min\{a_k,\ldots,a_\ell\}$, $n-q$ is the largest possible degree of an element 
$a_j\in\{a_{k+1},\ldots,a_\ell\}$ such that $|a_i|+|a_j|=n$ for some other element $a_i$.
Then $1\otimes(\alpha'_{k-1})_*(\iota_{n-2})=0$ in $(\bar E_{k-1})^{r}_{0,n-2}$ for $r>n-q$ 
since it cannot be in the image of any differential. Therefore in $(\bar E_{k-1})^{n-q}_{0,n-2}$ 
$$
\begin{array}{rcl}
1\otimes(\alpha'_{k-1})_*(\iota_{n-2})&=&\csum{|a_i|=q,|a_j|=n-q}{}{c''_{ij}(1\otimes\bar u_i\bar u_j)}\\
&=&\bar d^{n-q}(\zeta'').
\end{array}
$$

Since $q<n-q$, no nonzero element in
$(E_{k-1})^{n-q}_{n-q,q-1}$ is in the image of the differential $d^{n-q}$.
Likewise, no nonzero element $(\bar E_{k-1})^{n-q}_{n-q,q-1}$ is in the image of the differential $\bar d^{n-q}$.
Since $(E_{k-1})^{\infty}_{n-q,q-1}=\{0\}$ and $(\bar E_{k-1})^{\infty}_{n-q,q-1}=\{0\}$, the differentials
$$
\bar d^{n-q}\colon\seqm{(\bar E_{k-1})^{n-q}_{n-q,q-1}}{}{(\bar E_{k-1})^{n-q}_{0,n-2}}
$$
$$
d^{n-q}\colon\seqm{(E_{k-1})^{n-q}_{n-q,q-1}}{}{(E_{k-1})^{n-q}_{0,n-2}}
$$
must both be injections. 
Now because $\bar d^{n-q}(\zeta'')=1\otimes(\alpha'_{k-1})_*(\iota_{n-2})$ in $(\bar E_{k-1})^{n-q}_{0,n-2}$,
and $\gamma^{n-q}(1\otimes(\alpha'_{k-1})_*(\iota_{n-2}))=0$, then
$$
\gamma^{n-q}(\zeta'')=\csum{|a_i|=q,|a_j|=n-q}{}{c''_{ij}(a_j\otimes u_i)}=0.
$$

Now suppose $q=m_k$, where again we recall $q=|a_k|$.
In this case the projection of $1\otimes(\alpha'_{k-1})_*(\iota_{n-2})$ generates the kernel of 
$\gamma^r$ as in equation~(\ref{EKer2}).
Then by naturality of the spectral sequences $\zeta''$ generates the kernel of
$$
\gamma^{n-q}=\gamma^{n-m_k}\colon\seqm{\bar E^{n-m_k}_{n-m_k,m_k-1}}{}{E^{n-m_k}_{n-m_k,m_k-1}}.
$$ 
But as we saw before, the kernel of this is also generated by $\zeta$,
so we must have $\zeta''=b\zeta$ for some integer $b$ prime to $p$. 
Comparing coefficients in equations~(\ref{EZeta}) and~(\ref{EZeta1}), we set $b_{|a_k|}=b$, 
and we have $c''_{ij}=b_{|a_k|}(-1)^{m_k}c_{ij}$ for $i$ and $j$ such that $k\leq i<j$, $|a_i|=m_k$, and $|a_j|=n-m_k$.
Therefore equation~(\ref{EImage}) holds in this case.

Next suppose $q<m_k$. 
By part $(iv)$ of Proposition~\ref{L1a}
$$
\gamma^{n-q}\colon\seqm{\bar E^{n-q}_{n-q,q-1}}{}{E^{n-q}_{n-q,q-1}}
$$ 
is an isomorphism. Since $\gamma^{n-q}(\zeta'')=0$, 
we must have $c''_{ij}=0$ for each of the coefficients of $\zeta''$.
Then we can choose $b_{|a_k|}=1$ for example, and the result follows as the previous case.
This finishes the induction.

\end{proof}

\begin{lemma}
\label{LADJ}
Let $V=F\{x_1,\ldots,x_k\}$ be a vector space over a field $F$, 
and $T(V)$ the tensor algebra generated by $V$. 
For an odd integer $N$, consider a nonzero homogeneous element in $T(V)$ 
$$
\xi=\csum{|x_i|+|x_j|=N}{}{b_{ij} x_i x_j}
$$
satisfying the condition $b_{ij}\neq 0$ if and only if $b_{ji}\neq 0$,
and let $I$ be the two-sided ideal generated by $\xi$. 

Take a nonzero homogeneous element in $T(V)$
$$
u=\csum{|x_i|=M}{}{a_i x_i}
$$
for some integer $M$ and integers $a_i$.
Then for any element $w\in T(V)$, $wu\in I$ if and only if $w\in I$. 

\end{lemma}

\begin{proof}
We will say an element $w\in T(V)$ has \emph{length} $l$ if it 
is a linear combination of monomials in $T(V)$ of length at most $l$.
This gives a filtration of $T(V)$ by length.

Since $w\in I$ implies $wu\in I$, it remains to show that
$w\nin I$ implies $wu\nin I$. 
The proof is by induction on length of elements in $T(V)$. 
Assume $w\nin I$ implies $wu\nin I$ for all $w\in T(V)$ of length $l$. 
The base case $l=0$ is clearly true since $u$ is not in $I$.
The case $l=1$ is also obviously true since $\xi$ must have summands 
$b_{ij}x_i x_j$ with $b_{ij}\neq 0$ and $|x_j|\neq M$,
which $wu$ cannot have 
(this follows from the fact that $\xi\neq 0$, that $N$ is odd, 
and the property $b_{ij}\neq 0$ if and only if $b_{ji}\neq 0$).  

Consider a nonzero $w\in T(V)$ of length $l+1$ such that $w\nin I$.
Let us assume $wu\in I$. Using the inductive assumption we will
show this leads to a contradiction.
 
Note $wu$ has length $l+2$ and and $\xi$ has length $2$. 
Since $wu\in I$, we can write  
$$
wu=\csum{j}{}{v_j x_j}+v\xi,
$$ 
with $v$ of length $l$, and each $v_j$ is some (possibly zero) element in $I$ length $l+1$.

Observe $v\neq 0$, for otherwise $v\xi=0$, 
and so we would have $v_j=0$ whenever $|x_j|\neq M$,
and $v_j=a_j w$ whenever $|x_j|=M$, which would imply $w\in I$.
Expanding $wu$ and $v\xi$, and comparing like terms,
\begin{equation*}
v_j x_j+\csum{i}{}{v(b_{ij} x_i x_j)} = 
\begin{cases}
w (a_j x_j)&\mbox{ if }|x_j|=M;\\
0&\mbox{ if }|x_j|\neq M. 
\end{cases}
\end{equation*} 
Thus
\begin{equation*}
v_j+v y_j = 
\begin{cases}
a_j w&\mbox{ if }|x_j|=M;\\
0&\mbox{ if }|x_j|\neq M, 
\end{cases}
\end{equation*} 
where $y_j=\csum{i}{}{b_{ij}x_i}$. 

Take $j$ so that $|x_j|=M$ and $a_j\neq 0$.
Since $v_j+v y_j=a_j w$, $v_j\in I$, and $w\nin I$, 
it follows that $v y_j\nin I$.
Therefore $v\nin I$. 

Since $\xi\neq 0$, there exist $i$ and $j$ such that $b_{ij}x_i x_j\neq 0$,
implying $b_{ij}\neq 0$, and in turn $b_{ji}\neq 0$.
If it happens that $|x_j|=M$, we can instead consider the summand $b_{ji}x_j x_i\neq 0$, 
where we have $|x_i|\neq M$ since $|x_i|+|x_j|=N$ and $N$ is odd.

In any case, we can choose $i$ and $j$ such that $b_{ij}\neq 0$ and $|x_j|\neq M$.
Then $y_j\neq 0$, and since $v\neq 0$, $v y_j\neq 0$.  
Now because $v\nin I$ and $v$ is of length $l$, by our inductive assumption $v y_j\nin I$.
But $v_j+v y_j=0$ and $v_j\in I$, so $v y_j=-v_j$ is a nonzero element in $I$, a contradiction.
Therefore $wu\nin I$, which finishes our induction.
 
\end{proof}

\begin{theorem}
\label{P0}
Let $P$ be as in the introduction to this section. Assume the following condition holds true:
\begin{itemize}
\item[($\ast$)] there exist elements $a,b\in H^*(P;\zmodp)$ such that $0<|a|<|b|<n$, $|a|+|b|=n$, and the cup product
$ab\in H^n(P;\zmodp)$ is nonzero.
\end{itemize}
Then there is a Hopf algebra isomorphism 
$$
H_*(\Omega P;\zmodp)\cong T(\bar u_1,\ldots, \bar u_\ell)/I,
$$ 
where $I$ is the two-sided ideal of $H_*(\Omega \bar P;\zmodp)\cong T(\bar u_1,\ldots, \bar u_\ell)$
generated by the degree $n-2$ element 
$$
\chi=\csum{m\leq s\leq\ell}{}{b_s\kappa_s}
$$
as described in Proposition~\ref{L1}. 
Moreover, the looped inclusion \seqm{\Omega\bar P}{\Omega i}{\Omega P} induces
a map on mod-$p$ homology modelled by the canonical map 
\seqm{T(\bar u_1,\ldots, \bar u_\ell)}{}{T(\bar u_1,\ldots, \bar u_\ell)/I}.
\end{theorem}

\begin{proof}
To avoid confusing notation, 
we will write monomials in $T(\bar u_1,\ldots, \bar u_\ell)$ without the tensor product symbol.
By Proposition~\ref{L1}, 
the element $\chi\in H_*(\Omega\bar P;\zmodp)\cong T(\bar u_1,\ldots, \bar u_\ell)$ is in the image of the map
$$
(\Omega\alpha')_*\colon\seqm{H_{n-2}(S^{n-2};\zmodp)}{}{H_{n-2}(\Omega\bar P;\zmodp)}
$$
induced by the adjoint $\alpha'$ of the attaching map $\alpha$. Thus $\chi$ is a primitive element, 
and $(\Omega i)_*(\chi)=0$ in $H_*(\Omega P;\zmodp)$, where $i$ is the inclusion \seqm{\bar P}{i}{P}.  

Let $A$ be the quotient algebra of the tensor algebra 
$T(\bar u_1,\ldots, \bar u_\ell)$ modulo the two-sided ideal generated by the element $\chi$. 
Then $A$ is a Hopf algebra because $\chi$ is primitive. Since $(\Omega i)_*(\chi)=0$ in $H_*(\Omega P;\zmodp)$, 
the Hopf algebra map $\hat\theta=(\Omega i)$ factors through Hopf algebra maps
\begin{equation}
\label{DFactor}
\diagram
T(\bar u_1,\ldots, \bar u_\ell)\dto^{\hat\theta}\rto^{} & A\dlto^{\theta}\\
H_*(\Omega P;\zmodp), 
\enddiagram
\end{equation}
where the Hopf algebra map $\theta$ is defined by $\theta(\bar u_i)=u_i$. 

We let $m'$ be the smallest integer $m\leq m'\leq \floor{\frac{n}{2}}$ 
such that there is a $c_{ij}$ prime to $p$ for some $i\leq j$,
with $i$ satisfying $|a_i|=m'$. 
By condition ($\ast$) such an integer $m'$ exists, and $m'<n-m'$. 

Consider differential bigraded left $A$-modules
$$
\hat E^{2}_{*,*}=\cdots=\hat E^{m}_{*,*}=\zmodp\{1,a_1,\ldots,a_\ell,z\}\otimes A,
$$
the element
$$
\zeta=(-1)^{m'}\csum{|a_i|=m',|a_j|=n-m'}{}{c_{ij} (a_j\otimes \bar u_i)},
$$ 
with formal differentials $\hat d^{r}$ of bidegree $(-r,r-1)$ for $r\leq m$ given as follows. 
First define the morphism of left $T(\bar u_1,\ldots,\bar u_\ell)$-modules 
$$
\bar d^{r}\colon\seqm{\zmodp\{1,a_1,\ldots,a_\ell,z\}\otimes T(\bar u_1,\ldots,\bar u_\ell)}
{}{\zmodp\{1,a_1,\ldots,a_\ell,z\}\otimes T(\bar u_1,\ldots,\bar u_\ell)}
$$
by $\bar d^r=0$ when $r<m$, 
and respecting the left action of $T(\bar u_1,\ldots,\bar u_\ell)$ by assigning 
$$\bar d^m(x\otimes y)=(1\otimes y)\bar d^m(x\otimes 1),$$ 
where
$\bar d^{m}(1\otimes y)=0$;
$\bar d^{m}(a_i\otimes 1)=1\otimes \bar u_i$ whenever $|a_i|=m$, 
otherwise $\bar d^{m}(a_i\otimes 1)=0$;
and $\bar d^{m}(z\otimes 1)=\zeta$ when $m=m'$,
otherwise $\bar d^{m}(z\otimes 1)=0$.

Since $A$ is the quotient of $T(\bar u_1,\ldots,\bar u_\ell)$ subject to the relation $\chi\sim 0$, 
for $r\leq m$ the differential $\bar d^{r}$ induces a morphism $\hat d^{r}$ of $A$-modules
$$
\hat d^{m}\colon\seqm{\zmodp\{1,a_1,\ldots,a_\ell,z\}\otimes A}{}{\zmodp\{1,a_1,\ldots,a_\ell,z\}\otimes A}
$$ 
respecting the left action of $A$.

Next we define inductively for $r\geq m$
$$
\hat E^{r+1}_{*,*}=\frac{\ker \bracket{d^{r}\colon\seqm{E^{r}_{*,*}}{}{E^{r}_{*-r,*+r-1}}}}
{\mbox{Im } \bracket{d^{r}\colon\seqm{E^{r}_{*+r,*-r+1}}{}{E^{r}_{*,*}}}}.
$$
As before, we have formal differentials given as morphisms of left $A$-modules
$$
\hat d^{r+1}\colon\seqm{\hat E^{r+1}_{*,*}}{}{E^{r+1}_{*-(r+1),*+r}}
$$
respecting the left action of $A$,
and such that: $\hat d^{r+1}(1\otimes y)=0$; 
$\hat d^{r+1}(a_i\otimes 1)=1\otimes \bar u_i$ whenever $|a_i|=r+1$, 
otherwise $\hat d^{r+1}(a_i\otimes 1)=0$; 
and $\hat d^{r+1}(z\otimes 1)=\zeta$ if $r+1=m'$, otherwise $\hat d^{r+1}(z\otimes 1)=0$.

This gives a formal spectral sequence $\hat E=\{\hat E^r,\hat d^r\}$. 
We will need to verify that $\hat E^{\infty}_{*,*}=\{0\}$ for $(*,*)\not=(0,0)$, 
but let us assume that this is the case for now. 
We shall show by induction that the restiction $\theta\colon A_k\to H_k(\Omega P;\zmodp)$ 
of the Hopf algebra map $\theta$ is an isomorphism for each $k$. 

Let $E$ be mod-$p$ homology spectral sequence for the path fibration of $P$.
The morphism of Hopf algebras \seqm{A}{\theta}{H_*(\Omega P;\zmodp)} induces a morphism of spectral sequences
$$
\theta\colon\seqm{\hat E^r_{*,*}}{}{E^r_{*,*}}
$$
in the canonical way with $\theta(1\otimes \bar u_i)=1\otimes u_i$, 
$\theta(\bar a_i\otimes 1)=a_i\otimes 1$, and $\theta(z\otimes 1)=z\otimes 1$. 
Note $\seqm{\hat E^m_{0,*}}{\theta}{E^m_{0,*}}$ is just our map \seqm{A}{\theta}{H_*(\Omega P;\zmodp)},
and $A_q=H_q(\Omega P;\zmodp)=\{0\}$ for $0<q<m-1$.

Suppose \seqm{A_q}{\theta}{H_q(\Omega P;\zmodp)} is an isomorphism for $0<q<k$. 
This implies 
$\seqm{\hat E^r_{0,q}}{\theta}{E^r_{0,q}}$
is an isomorphism, and
$\seqm{\hat E^r_{i,q}}{\theta}{E^r_{i,q}}$ 
is an isomorphism when $q+r-1<k$.
Since $E^{\infty}_{*,*}=\{0\}$ and $\hat E^{\infty}_{*,*}=\{0\}$ when $(*,*)\not=(0,0)$,
for some sufficiently large $M>m$ the map $\seqm{\hat E^M_{0,k}}{\theta}{E^M_{0,k}}$ is an isomorphism.
By definition of spectral sequences, there is a commutative diagram of short exact sequences
\[\diagram
\hat E^{r-1}_{r-1,k-r+2}\rto^{\hat d^{r-1}}\dto^{\theta} & 
\hat E^{r-1}_{0,k}\rto^{proj.}\dto^{\theta} & 
\hat E^{r}_{0,k}\rto^{}\dto^{\theta} & 0\ddouble\\
E^{r-1}_{r-1,k-r+2}\rto^{d^{r-1}}& 
E^{r-1}_{0,k}\rto^{proj.}& 
E^{r}_{0,k}\rto^{}& 0.
\enddiagram\]
By induction the first vertical map is an isomorphism when $r>2$.
When $r=M$ the third vertical map is an isomorphism, 
and so the second vertical map is also an isomorphism.
Iterating this argument over $m\leq r<M$, we see that the map
$$
\theta\colon\seqm{A_k=\hat E^{m}_{0,k}}{}{E^{m}_{0,k}=H_k(\Omega P;\zmodp)}
$$ 
is an isomorphism. This completes the induction.

It remains to check that $\hat E^{\infty}_{*,*}=\{0\}$ for $(*,*)\not=(0,0)$.
Let $\bar E$ be mod-$p$ homology spectral sequence for the path fibration of $\bar P$. 
We have 
$$
\bar E^m_{*,*}\cong \bar E^2_{*,*}= H_*(\bar P;\zmodp)\otimes H_*(\Omega\bar P;\zmodp)
\cong \zmodp\{1,a_1,\ldots,a_\ell\}\otimes T(\bar u_1,\ldots,\bar u_\ell),
$$
and $\bar E^{\infty}_{*,*}=\{0\}$ when $(*,*)\neq (0,0)$.
The Hopf algebra map
$\seqm{H_*(\Omega\bar P;\zmodp)\cong T(\bar u_1,\ldots, \bar u_\ell)}{}{A}$ 
induces a morphism of spectral sequences
$$
\phi\colon \seqm{\bar E}{}{\hat E}
$$
in the canonical way with $\phi^2(1\otimes\bar u_i)=1\otimes \bar u_i$, $\phi^2(a_i\otimes 1)=a_i\otimes 1$.
Notice
$$
\phi^r\colon \seqm{\bar E^r_{i,j}}{}{\hat E^r_{i,j}}
$$ 
is an epimorphism when $i<n$, and is an isomorphism when $i<n$, $j<n-2$, and $r\leq m'$. 
The differentials 
\seqm{\hat E^{r}_{i,j}}{\hat d^{r}}{\hat E^{r}_{i-r,i+r-1}}
are zero for $r<i$ and $i<n$,
so when $i<n$ and $r<\min\{i,n-i\}$, we have projections
\begin{equation}
\label{EProj}
\seqm{\hat E^{r}_{i,j}}{}{\hat E^{r+1}_{i,j}}
\end{equation}   
that are isomorphisms.
Also, $\chi$ is nonzero in $\bar E^r_{0,n-2}$ for $r\leq n-m'$, 
and zero for $r>n-m'$ since $\bar d^{n-m'}(\zeta)=\chi$.

Note $\hat E^m_{k,l}=\{0\}$ when $0<k<m$, $k>n$, or $l<m-1$.
We will first consider those nonzero elements in $\hat E^m_{n,l}$ and $\hat E^m_{0,l}$ for $l\geq m-1$.

Take any nonzero $x\in \hat E^m_{n,l}$. Then $x=z\otimes w$ for some nonzero $w\in A$.
Note $\hat d^r(\hat E^r_{n,l})=\{0\}$ when $r<m'$, so we can project $x$ to $\hat E^{m'}_{n,l}$.
Pick $w'\in T(\bar u_1,\ldots,\bar u_\ell)$ such that $w'$ projects onto $w\in A$.
Since $w$ is nonzero in $A$, $w'$ is not in the two-sided ideal generated by $\chi$.
Take the element following element in $T(\bar u_1,\ldots,\bar u_\ell)$:
$$
\sigma'_j=w\bracket{\csum{|a_i|=m'}{}{c_{ij}\bar u_i}}.
$$
Let $\sigma_j\in A$ be the projection of $\sigma'_j$ onto $A$. We have 
\begin{align*}
\hat d^{m'}(x)&=(1\otimes w)\hat d^{m'}(z\otimes 1)=(1\otimes w)(\zeta)\\
&= (-1)^{m'}\csum{|a_i|=m',|a_j|=n-m'}{}{c_{ij} (a_j\otimes (w\bar u_i))}\\
&= (-1)^{m'}\csum{|a_j|=n-m'}{}{a_j\otimes\sigma_j}. 
\end{align*}
By condition ($\ast$) there are integers $k<l$, with $|a_l|=m'$ and $|a_k|=n-m'$, such that $c_{lk}$ is prime to $p$,
so the element $\sigma'_k$ is nonzero. 
Because $n$ is odd and $w'$ is not in the two-sided ideal generated by $\chi$,  
Lemma~\ref{LADJ} implies $\sigma'_k$ is not in the two-sided ideal generated by $\chi$. 
Therefore $\sigma_k\in A$ is nonzero, 
so $a_k\otimes\sigma_k\in\hat E^m_{*,*}=\zmodp\{1,a_1,\ldots,a_\ell,z\}\otimes A$ is nonzero,
and by this we see that $\hat d^{m'}(x)\in\hat E^m_{n-m',l+m'-1}$ is also nonzero.
By the projection isomorphisms~(\ref{EProj}), this implies $\hat d^{m'}(x)\in\hat E^{m'}_{n-m',l+m'-1}$ is nonzero,
and so $x$ does not survive to $\hat E^{m'+1}_{n,l}$.
Thus $\hat E^{\infty}_{n,l}=\hat E^{m'+1}_{n,l}=\{0\}$.

Now take $x\in \hat E^m_{0,l}$ for $l\geq m-1$. 
We can pick $x'\in\bar E^m_{0,l}$ so that $\phi^m(x')=x$.
Since $\bar E^{\infty}_{0,l}=\{0\}$, there exists an $\dot x\in \bar E^r_{*,*}$
for some $r\geq m$ such that $\bar d^r(\dot x)=x'$. 
Then in $\hat E^r_{0,l}$, 
$$
x=\phi^r(x')=\phi^r(\bar d^r(\dot x))=\hat d^r(\phi^r(\dot x)),
$$ 
and so $x=0$ in $\hat E^{r+1}_{0,l}$. Thus $\hat E^{\infty}_{0,l}=\{0\}$.

It remains to consider those elements in $\hat E^m_{k,l}$ when $m\leq k<n$. 
Because the elements in $\hat E^m_{k,0}$ for $m\leq k<n$ are transgressive, 
the differentials \seqm{\hat E^i_{k,l}}{d^i}{\hat E^i_{k-i,l+i-1}} 
are zero for $l\geq 0$ and $m\leq i< k$,
and so we might as well project to $\hat E^k_{k,l}$.

Suppose $x\in \hat E^k_{k,l}$ and $x\neq 0$.
We will show that $\hat d^k(x)\neq0$.
Hence $\hat E^{\infty}_{k,l}=\hat E^{k+1}_{k,l}=\{0\}$.
There are three subcases: $m\leq k<n-m'$, $k=n-m'$, and $n-m'<k<n$.
We assume $x\neq 0$ and $\hat d^k(x)=0$ to arrive at a contradiction.

Let us first consider the case $m\leq k<n-m'$. 
We can pick $x'\in\bar E^k_{k,l}$ such that $\phi^k(x')=x$. 
Then $\phi^k(\bar d^k(x'))=\hat d^k(x)=0$, 
and so inspecting the kernel of \seqm{\bar E^k_{0,k+l-1}}{\phi^k}{\hat E^k_{0,k+l-1}}, 
$y'=\bar d^k(x')\in \bar E^k_{0,k+l-1}$ must be a linear combination
$$
y'=\csum{i}{}{v_i\chi w_i},
$$
where $v_i$ and $w_i$ are monomials in $T(\bar u_1,\ldots,\bar u_\ell)$.
Since $x'$ is nonzero in $\bar E^k_{*,*}$, and $\bar E^{\infty}_{*,*}=\{0\}$ for $(*,*)\neq (0,0)$, 
$y'$ must also nonzero in $\bar E^k_{*,*}$, 
and so we might as well assume the monomials $v_i$ and $w_i$ are nonzero.
%Since $k<n-m'$, $\chi$ is nonzero in $\bar E^{n-m'}_{*,*}$,
%so $v\chi$ is also nonzero in $\bar E^{n-m'}_{*,*}$ for any nonzero monomial $v$.
Since $y'$ is in the image of $\bar d^k$ with $k<n-m'$, 
and (by definition of $m'$) $\chi$ is linear combination containing summands $c_{ij}u_i u_j$
with $|u_j|=n-m'$ and $c_{ij}$ prime to $p$,
each of the $w_i$'s must be a monomial of length at least $1$ of the form
$$
w_i=w'_i\bar u_{k_i}
$$ 
for some monomial $w'_i$ and $\bar u_{k_i}$ such that $|\bar u_{k_i}|=k$.
Since $\bar d^k(x')=y'$,
$$
x'=\csum{i}{}{a_{k_i}\otimes v_i\chi w'_i}.
$$   
But since $\chi$ is zero in $A$, each term $v_i\chi w'_i$ is as well, 
and so $x=0$ in $\hat E^k_{k,l}$, a contradiction.
Hence we must have $\hat d^k(x)\neq 0$.

Now consider the case $n-m'<k<n$.
As in the previous case, we can write $y'$ as
$$
y'=\csum{i}{}{v_i\chi w_i}.
$$
Here $w\chi$ is also zero in $\bar E^{k}_{*,*}$ for any monomial $w$, since 
$$
w\chi=\bar d^{n-m'}\bracket{(-1)^{m'}\csum{|a_i|=m',|a_j|=n-m'}{}{c_{ij} (a_j\otimes w\bar u_i)}}
$$
Therefore each monomial $w_i$ is nonzero of length at least $1$,
(for otherwise $0=y'=\bar d^k(x')$, which implies $x'=0$), 
with
$w_i=w'_i\bar u_{k_i}$ for some monomial $w'_i$ 
and $\bar u_{k_i}$ such that $|\bar u_{k_i}|=k$.  
As before this implies $x=0$, a contradiction.
Thus $\hat d^k(x)\neq 0$.

Finally let us consider $k=n-m'$. 
In this case we can write
$$
y'=\csum{i}{}{v_i\chi w_i}+\csum{i}{}{y_i\chi}
$$
for some nonzero monomial $v_i$, and nonzero monomial $w_i$ of length at least $1$.
As before, we must have $w_i=w'_i\bar u_{k_i}$ for some $w'_i$ 
and $\bar u_{k_i}$ such that $|\bar u_{k_i}|=n-m'$. 
Let $\zeta'\in\bar E^{n-m'}_{n-m',m'-1}$ be the element satisfying
$\phi^{n-m'}(\zeta')=\zeta$.
Notice that in $\bar E^{n-m'}_{0,n-2}$ 
we have $\bar d^{n-m'}(b\zeta')=\chi$ for some integer $b$ prime to $p$. 
Thus 
$$
x'=\csum{i}{}{(a_{k_i}\otimes v_i\chi w'_i)}+b\csum{i}{}{(1\otimes y_i)\zeta'}.
$$
Since $\chi$ is zero in $A$, 
$$
x=\phi^{n-m'}(x')=\phi^{n-m'}(b\csum{i}{}{(1\otimes y_i)\zeta'})=b\csum{i}{}{(1\otimes y_i)\zeta}.
$$
But in $\hat E^{m'}_{*,*}$ we have $\bar d^{m'}(z\otimes 1)=\zeta$, 
so $\bar d^{m'}(z\otimes y_i)=(1\otimes y_i)\zeta$. 
So because $m'<n-m'$, $(1\otimes y_i)\zeta$ is zero in $\hat E^{n-m'}_{*,*}$, 
and it follows that $x=0$, a contradiction. Hence $\hat d^k(x)\neq 0$.

\end{proof}

\section{Some additional structure on the mod-$p$ Cohomology ring}
\label{SStruct}
The non-trivial action of Bockstein operations impose the restrictions seen in Proposition~\ref{P1}. 
Later we will see they are necessary for Theorem~\ref{Main} to be true in general.
Recall that the mod-$p$ Bockstein operations $\beta_r\colon\seqm{H_*(X)}{}{H_{*-1}(X)}$
are derivations with respect to the homology multiplication induced by an $H$-space structure on $X$. 
That is,
$$
\beta_r(xy)=\beta_r(x)y+(-1)^{|x|}x\beta_r(y).
$$

\begin{proof}[Proof of Proposition~\ref{P1}]
Take the attaching map \seqm{S^{2n-2}}{\alpha}{\bar W} and its adjoint $\alpha'$. 
By Proposition~\ref{L1}, in $H_{2n-3}(\Omega \bar W;\zmodp)$ we have
$$
\alpha'_*(\iota_{2n-3})=\csum{i,j}{}{(-1)^{n-1}a_{ij}[\bar u_i,\bar v_j]}.
$$ 
for some generator $\iota_{2n-3}\in H_{2n-3}(S^{2n-3};\zmodp)$. 
As sets we have 
$$\paren{r_1,r_2,\ldots,r_{k_1}}=\paren{s_1,s_2,\ldots,s_l}$$
for some $s_1<s_2<\cdots<s_l$ and $l\leq k_1$. Since $\bar u_i$ and $\bar v_i$ are 
the transgressions of $x_i$ and $y_i$, 
$\beta_{r_i}(\bar v_{i})=\bar u_{i}$ for $1\leq i\leq k_1$, 
and $\beta_{r}(\bar v_{i})=0$ for each $r>0$ and $k_1<i\leq k$.
Therefore 
$\beta_{s_t}([\bar u_i,\bar v_j])=0$ whenever $j>k_1$ or $s_t\neq r_j$. Then because 
$\beta_{s_t}(\alpha'_*(\iota_{2n-3}))=\alpha'_*(\beta_{s_t}(\iota_{2n-3}))=\alpha'_*(0)=0$
for each $t$,
$$
\begin{array}{rcl}
0&=&\csum{t=1}{l}{\beta_{s_{t}}(\alpha'_*(\iota_{2n-3}))}\\
&=&\csum{t=1}{l}{\beta_{s_{t}}\bracket{\csum{i,j}{}{(-1)^{n-1}a_{ij}[\bar u_i,\bar v_j]}}}\\
&=&\csum{j\leq k_1\mbox{, }i}{}{(-1)^{n-1}a_{ij}\beta_{r_j}([\bar u_i,\bar v_j])}\\
&=&(-1)^{n-1}\csum{j\leq k_1\mbox{, }i}{}{a_{ij}[\bar u_i,\bar u_j]}\\
&=&(-1)^{n-1}\bracket{\csum{i=1}{k_1}{a_{ii}[\bar u_i,\bar u_i]}+
\csum{j<i\leq k_1}{}{(a_{ij}-(-1)^{n}a_{ji})[\bar u_i,\bar u_j]}+
\csum{j\leq k_1,i>k_1}{}{a_{ij}[\bar u_i,\bar u_j]}}.\\
\end{array}
$$
When $n$ is odd it follows that $a_{ii}=0$ and $a_{ij}+a_{ji}=0$ whenever $j< i\leq k_1$, 
and $a_{ij}=0$ whenever $1\leq j\leq k_1$ and $k_1<i\leq k$. 
Namely $B_{z^*}$ is skew symmetric and $C_{z^*}=0$. 
When $n$ is even $a_{ij}-a_{ji}=0$ and $[\bar u_i,\bar u_i]=0$,
so there is no restriction on the $a_{ii}$'s. 
In this case $B_{z^*}$ is symmetric, and likewise $C_{z^*}=0$.
\end{proof}

\section{The Effect of Looping in Rank One}
\label{RankOne}

Fix a class $[V]\in \mathcal{T}^{p}_{1,n}$ for $p$ an odd prime.
Again we assume that all our spaces are localized at $p$. 
$H_*(V;\zmodp)$ is generated by $x$, $y$, and $z$, where $|x|=n-1$, $|y|=n$, and $|z|=2n-1$.
If $\beta_r(y)=x$ for some $r>0$, then we can and will take $V\in [V]$ 
so that $(2n-2)$-skeleton $\bar V$ of $V$ is the Moore space space $\M{n}{r}$.
Similarly, when $\beta_r(y)=0$ for every $r>0$, $V\in [V]$ can be taken so that $\bar V=S^{n-1}\vee S^n$.

Let $u$ and $v$ in $H_*(\Omega V;\zmodp)$ be the transgressions of $x$ and $y$ respectively, with $|v|=n-1$ and $|u|=n-2$. 
The following are consequences of Proposition~\ref{P0} and Proposition~\ref{L1}.

\begin{corollary}
\label{C1}
Take $[V]\in \mathcal{T}^{p}_{1,n}$ with $n\geq 3$. 
Then $H_*(\Omega V;\zmodp)\cong T(u)\otimes T(v)$ as Hopf algebras.~$\qqed$
\end{corollary}

\begin{corollary}
\label{C2}
Take $[V]\in \mathcal{T}^{p}_{1,n}$ with $n\geq 3$.
Let $\alpha\colon S^{2n-2}\to \M{n}{r}$ be the attaching map for the $V$, 
and $\alpha'\colon S^{2n-3}\to \Omega \M{n}{r}$ be the adjoint map of $\alpha$. Then 
$$
\alpha'_*(\iota_{2n-3})=[u,v].
$$~$\qqed$
\end{corollary}

The following lemma is a special case of Barratt's work on growth of homotopy exponents~\cite{Barratt}, 
or Theorem~$4.1$ in~\cite{GW}. 

\begin{proposition}
\label{Lexp}
Fix some $n\geq 4$, and let $C$ be a finite wedge of Moore spaces $\M{n}{r_i}$. 
Let $s=\max_{i}\paren{r_i}$. Then 
$$
p^{s+k}\pi_{i}(C)=0\mbox{ if }\,i\leq p^{k+1} (n-2).
$$~$\qqed$ 
\end{proposition}

Let $\MS{2m-1}{r}$ be the homotopy theoretic fiber of the degree $p^r$ map \seqm{S^{2m-1}}{p^r}{S^{2m-1}}.
Recall from~\cite{CMN} there exists map 
\begin{equation}
\label{EMaph}
h\colon\seqm{\MS{2m-1}{r}}{h}{\Omega\M{2m}{r}}
\end{equation}
that is modelled on mod-$p$ homology by mapping $H_*(\MS{2m-1}{r};\zmodp)$ 
isomorphically onto the left $T(u)$-submodule of $H_{*}(\Omega\M{2m}{r};\zmodp)\cong T(u,v)$.
This map has a left homotopy inverse \seqm{\Omega\M{2m}{r}}{h^{-1}}{\MS{2m-1}{r}} 
that induces a map on mod-$p$ homology modelled by the abelianization map \seqm{T(u,v)}{}{S(u,v)}.

\begin{lemma}
\label{L2}
Let $[V]\in \mathcal{T}^{p}_{1,2m}$ with $m>2$ and $\beta_r(y)=x$ for some $r>0$.
Then $v^2$ is spherical in $H_*(\Omega V;\zmodp)$.
\end{lemma}
\begin{proof}

Let \seqm{S^{4m-2}}{\alpha}{\M{2m}{r}} be the attaching map for $V$. 
By Proposition~\ref{Lexp}, $[\alpha]$ has order $p^r$ in $\pi_{4m-2}(\M{2m}{r})$. 
Thus $\alpha$ extends to a map \seqm{\M{4m-1}{r}}{\bar\alpha}{\M{2m}{r}}. 
By taking the adjoint of $\bar\alpha$, we have the map
$$
\bar\alpha'\colon \seqm{\M{4m-2}{r}}{}{\Omega \M{2m}{r}}
$$
which induces 
$$
\bar\alpha'_*\colon \seqm{\bar H_*(\M{4m-2}{r};\zmodp)}{}{H_*(\Omega \M{2m}{r};\zmodp)}.
$$

Let $u'\in H_{4m-3}(\M{4m-2}{r};\zmodp)$ and $v'\in H_{4m-2}(\M{4m-2}{r};\zmodp)$ be a basis with $\beta_r(v')=u'$.
Since $\bar\alpha'$ restricted to $S^{4m-3}$ is $\alpha'$, we have
$$
\bar\alpha'_*(u')=\pm [u,v].
$$
We may assume that $\bar\alpha'_*(u')=[u,v]$ because otherwise we can replace $u'$ and $v'$ to be $-u'$ and $-v'$. 

When $m>2$, $|u^3|=3(2m-2)=6m-6>|v^2|=2(2m-1)=4m-2$, which implies $2m>4$.
Then $H_{4m-2}(\Omega \M{2m}{r};\zmodp)$ is a $1$-dimensional vector space with a basis given by $v^2$,
implying $\alpha'_*(v')=kv^2$ for some $k$. Since
$$
k[u,v]=\beta_r(kv^2)=\beta_r(\bar\alpha'_*(v'))=\bar\alpha'_*(\beta_r(v'))=\bar\alpha'_*(u')=[u,v],
$$
we have $k=1$. Therefore $\bar\alpha'_*(v')=v^2$. 

Consider the composite
$$
f\colon\seqmm{\M{4m-2}{r}}{\bar\alpha'}{\Omega \M{2m}{r}}{}{\Omega V}.
$$
By Corollary~\ref{C1} there is a Hopf algebra isomorphism $H_*(\Omega V;\zmodp)\cong T(u)\otimes T(v)$, 
and the $H$-map \seqm{\Omega\M{2m}{r}}{}{\Omega V} induces a map on mod-$p$ homology
modelled by the algebra map \seqm{T(u,v)}{}{T(u)\otimes T(v)} sending $u$ to $u$ and $v$ to $v$. 
Thus $f_*(\iota_{4m-2})=v^2$.

Now $f$ factors through a map 
$$
\bar f\colon\seqm{S^{4m-2}}{}{\Omega V}, 
$$
because its restriction to the bottom sphere $S^{4m-3}$ is the adjoint 
of the null homotopic map \seqmm{S^{4m-2}}{\alpha}{\M{2m}{r}}{}{V}, 
and we have $\bar f_*(\iota_{4m-2})=f_*(\iota_{4m-2})=v^2$.

\end{proof}

\begin{theorem}
\label{T1}
Take $[V]\in\mathcal{T}^{p}_{1,2m}$ with $m>2$ and $\beta_r(y)=x$ for some $r>0$.
Then
$$\Omega V\simeq \MS{2m-1}{r}\times\Omega S^{4m-1}.$$
\end{theorem}
\begin{proof}

Consider the composite
$$\phi:\seqmm{\MS{2m-1}{r}}{h}{\Omega\M{2m}{r}}{}{\Omega V},$$
where the last map is the looped inclusion. The map \seqm{\MS{2m-1}{r}}{h}{\Omega\M{2m}{r}} is modelled on mod-$p$ homology 
by mapping $H_*(\MS{2m-1}{r};\zmodp)$ isomorphically onto the left $T(u)$-submodule of 
$H_{*}(\Omega\M{2m}{r};\zmodp)\cong T(u,v)$ with basis \paren{1,v}, where $|v|=n-1$ and $|u|=n-2$. 
Also, by Theorem~\ref{P0}, there is a Hopf algebra isomorphism
$$
H_*(\Omega V;\zmodp)\cong T(u)\otimes T(v)
$$ 
and the $H$-map \seqm{\Omega\M{2m}{r}}{}{\Omega V} induces a map on mod-$p$ homology
modelled by the algebra map \seqm{T(u,v)}{}{T(u)\otimes T(v)} sending $u$ to $u$ and $v$ to $v$.
It follows that $\phi_*$ is modelled by an isomorpism onto the left $T(u)$-submodule of 
$T(u)\otimes T(v)$ with basis \paren{1,v}. 

Now consider the map \seqm{S^{4m-2}}{\alpha'}{\Omega V} from the proof of Lemma~\ref{L2} which makes the class 
$v^{2}\in H_*(\Omega V;\zmodp)$ spherical. 
Since $\Omega V$ is an $H$-space, by the universal property of the James construction for $\Omega S^{4m-1}$ 
$\alpha'$ extends to an $H$-map \seqm{\Omega S^{4m-1}}{\theta}{\Omega V}. 
Then $\theta_*$ is modelled on mod-$p$ homology by mapping $T(\iota_{4m-2})$ onto the subalgebra of 
$T(u)\otimes T(v)$ generated by $v^2$. 
    
One now sees that the product 
$$\seqmm{\MS{2m-1}{r}\times \Omega S^{4m-1}}{\phi\times\theta}{\Omega V\times\Omega V}{mult.}{\Omega V}$$ 
induces an isomorphism on mod-$p$ homology, thus is a homotopy equivalence.
\end{proof}

The following theorem is probably well known. 

\begin{theorem}
\label{T1b}
Take $[V]\in\mathcal{T}^{p}_{1,n}$ with $\beta_r(y)=0$ for each $r>0$.
Then
$$
\Omega V\simeq \Omega S^{n-1}\times\Omega S^{n}.
$$
\end{theorem}
\begin{proof}
We take $V\in[V]$ so that $\bar V=S^{n-1}\vee S^n$.
Recall for general spaces $X$ and $Y$, the looped inclusion 
$\seqm{\Omega(X\vee Y)}{}{\Omega (X\times Y)=\Omega X\times\Omega Y}$
has a right homotopy inverse.
Thus for $X=S^{n-1}$ and $Y=S^n$ we have a right homotopy inverse 
$$
\seqm{\Omega S^{n-1}\times \Omega S^n}{s}{\Omega(S^{n-1}\vee S^n)}.
$$
On mod-$p$ homology, $s_*$ is modelled by the inclusion of Hopf algebras
\seqm{T(u')\otimes T(v')}{}{T(u',v')}, where $|x'|=n-2$ and $|y'|=n-1$.
Since $\Omega(i')_*$ is an algebra map, it is clear that the composite
$$
\seqmm{\Omega S^{n-1}\times \Omega S^n}{s}{\Omega(S^{n-1}\vee S^n)}{i}{\Omega V}
$$
induces an isomorphism in mod-$p$ homology, and so is a homotopy equivalence.
\end{proof}

\section{Higher ranks}
\label{HighRank}

Throughout this section we fix some class $[W]\in\mathcal{T}^{p}_{k,2m}$ with $k\geq 2$, and $m>2$,
with all spaces localized at an odd prime $p$.
 
We recall the properties for $W$ described in Section~\ref{SModp}.
The generators $x_1,\ldots,x_k$ and $y_1,\ldots,y_k$ will denote the basis for $H_{2m-1}(W;\zmodp)$
and $H_{2m}(W;\zmodp)$ dual to the mod-$p$ cohomology basis that we gave earlier, 
while $u_1,\ldots,u_k\in H_{2m-2}(\Omega V;\zmodp)$ and $v_1,\ldots,v_k\in H_{2m-1}(\Omega V;\zmodp)$ 
will denote the transgressions of the $x_i$'s and $y_j$'s. We have $\beta_{r_i}(y_i)=x_i$
for some integers $r_1,\ldots,r_{k_1}$, and integer $0\leq k_1\leq k$.
For convenience we take $W\in[W]$ so that the homotopy equivalence in equation~(\ref{ESplit}),
corresponding to our choice of basis above, is a homeomorphism.

Recall the $k\times k$ \zmodp-matrix $A_{z^*}=(a_{ij})$ associated with the cup product 
structure of $H^*(W;\zmodp)$ with respect to our choice of basis. We have $A_{z^*}$ is nonsingular. 
By Proposition~\ref{P1} the $k_1\times k_1$ matrix $B_{z^*}$ in the block
partition of $A_{z^*}$ (equation~\ref{EMatrix}) is symmetric, and the $k_2\times k_1$
matrix $C_{z^*}$ is zero. In paricular $a_{ij}=0$ for $k_1<i\leq k$, 
and $a_{ij}=a_{ji}$ whenever $1\leq i\leq k_1$. 

Let us assume $k_1\geq1$ for now.
We may as well assume our mod-$p$ homology basis has been ordered so that
$$
r_1=\max\paren{r_1,\ldots, r_{k_1}}.
$$ 
Since $A_{z^*}$ is nonsingular, there must exist an integer $i>1$ such that $a_{i1}\neq0$ 
whenever $a_{11}=0$. 
If this is the case, for convenience we assume our mod-$p$ homology basis 
corresponding to the splitting of $\bar W$ has been ordered so that $i=2$.

We will construct a certain map 
\begin{equation}
\label{EMainMap}
q\colon\seqm{W}{}{V} 
\end{equation}
which will be used in the upcoming proofs.
Here $[V]\in\mathcal{T}^{p}_{1,2m}$ with $\bar V=\M{2m}{r_1}$, and $q_*$ is nonzero in degree $2n-1$, 
and is nonzero for some choice of degree $2m-1$ and degree $2m$ generators. 
The restrictions on the matrix $A_{z^*}$ mentioned above will be necessary in order for $q$ to exist in general.
A similar map is constructed for the special case $k_1=0$.
This construction of will depend on a few seperate cases, again assuming $k_1\geq 1$:
\medskip
\begin{itemize}
\item[(1)] $a_{11}\neq0$.
\item[(2)] $a_{11}=0$: 
Since $A_{z^*}$ is nonsingular, there is an integer $i>1$ such that $a_{i1}\neq0$. 
We must have $i\leq k_1$, because $a_{ij}=0$ when $i>k_1$.
So $i$ corresponds to a Moore space \M{2m}{r_i} in the splitting of $\bar W$.
We consider three subcases:
\begin{itemize}
\item[(a)] $r_1=r_2$ and $a_{22}\neq 0$;
\item[(b)] $r_1=r_2$ and $a_{22}=0$;
\item[(c)] $r_1>r_2$.
\end{itemize}  
\end{itemize}

If the first case holds, let $\hat W=\bar W/\M{2m}{r_1}$. One may notice that the quotient 
$V=W/\hat W,$ which extends the quotient $\bar W / \hat W=\M{2m}{r_1}$, 
has its homotopy type in $\mathcal{T}^{p}_{1,2m}$.
Otherwise when part $(a)$ of the second case holds, let us fix 
$\hat W=\bar W/\M{2m}{r_2}$ and $V=W/\hat W$. 
In either case we set \seqm{W}{q}{V} as the respective quotient map.

Now consider parts $(b)$ and $(c)$ of the second case. $A_{z^*}$ being symmetric implies $a_{21}=a_{12}$. 
Setting $\hat W=\bar W/(\M{2m}{r_1}\vee \M{2m}{r_2})$, 
let $V'$ denote the quotient $W/\hat W$ and $\seqm{W}{q'}{V'}$ the corresponding quotient map. 
Set $t=r_1-r_2\geq 0$, and take the map \seqm{\M{2m}{r_2}}{\zeta}{\M{2m}{r_1}} as the induced map of cofibers
in the cofibration diagram
\[\diagram
S^{2m-1}\rto^{p^{r_2}}\ddouble & S^{2m-1}\rto^{}\dto^{p^{t}} & \M{2m}{r_2}\dto^{\zeta}\\
S^{2m-1}\rto^{p^{r_1}} & S^{2m-1}\rto^{} & \M{2m}{r_1}.
\enddiagram\]
Let $V$ be the pushout in the pushout diagram
\[\diagram
\M{2m}{r_1}\vee\M{2m}{r_2}\dto^{\mathbbm{1}\vee\zeta}\rto^{} & V'\dto^{}\\
\M{2m}{r_1}\rto^{} & V,
\enddiagram\]
where the horizontal maps are inclusions. Let $q$ be the composite 
$$
q\colon\seqmm{W}{q'}{V'}{}{V}.
$$
Notice $q$ extends the composite 
$$
\seqmm{\bar W}{}{\bar W/\hat W=\M{2m}{r_1}\vee \M{2m}{r_2}}{\mathbbm{1}\vee\zeta}{\M{2m}{r_1}}, 
$$
and \seqm{H^n(\M{2m}{r_2};\zmodp)}{\zeta^*}{H^n(\M{2m}{r_1};\zmodp)} is an isomorphism when $n=2m$, and multiplication by $p^t$ when
$n=2m-1$ (hence trivial when $t>0$). Thus 
$$
q^*(x^*)=x_{1}^*+p^{t}x_{2}^*,
$$ 
and
$$
q^*(y^*)=y_{1}^*+y_{2}^*
$$ 
for some generators $x^*$ and $y^*$ in $H^{2m-1}(V;\zmodp)$ and $H^{2m}(V;\zmodp)$. For part $(c)$, when $t=r_1-r_2>0$, we have
$$
q^*(x^*y^*)=(x_{1}^*+p^{t}x_{2}^*)(y_{1}^*+y_{2}^*)=(a_{11}+a_{21}+p^{t}a_{12}+p^{t}a_{22})z^*=a_{21} z^*.
$$
Therefore $x^*y^*=a_{21} e^*$ for some generator $e\in H^{4m-1}(V;\zmodp)\cong\zmodp$. Since we are assuming $a_{21}\neq 0$,
the homotopy type of $V$ is in $\mathcal{T}^{p}_{1,2m}$. For part $(b)$, when $t=r_1-r_2=0$ and $a_{22}= 0$,
$$
q^*(x^*y^*)=(a_{21}+a_{12}) z^*=2(a_{21}) z^*,
$$
and so the homotopy type of $V$ is in $\mathcal{T}^{p}_{1,2m}$ for this case as well.

Finally we consider the construction of the map \seqm{W}{q}{V} for the case $k_1=0$.
This time $[V]\in\mathcal{T}^{p}_{k,2m}$ satisfies $\bar V=S^{n-1}\vee S^n$.
The construction is straightforward. 
The nonsingular $A_{z^*}$ must have $a_{i1}\neq 0$ for some $i$. 
Assume our basis has been ordered so that $i=1$. 
Let $\Hat W=\bar W/(S^{n-1}\vee S^n)$, 
where the spheres $S^{n-1}$ and $S^n$ in the splitting of $\bar W$
correspond to the generators $x_1$ and $y_1$. 
Now let $V=W/\Hat W$, and $q$ be the corresponding quotient map.

We shall let $x^*\in H^{2m-1}(\bar V;\zmodp)=H^{2m-1}(V;\zmodp)$ and $y^*\in H^{2m}(\bar V;\zmodp)=H^{2m}(V;\zmodp)$ 
be generators with $\beta_r(x^*)=y^*$, $x\in H_{2m-1}(\bar V;\zmodp)=H_{2m-1}(V;\zmodp)$ 
and $y\in H_{2m}(\bar V;\zmodp)=H_{2m}(V;\zmodp)$ be their homology duals, and $u\in H^{2m-2}(\Omega\bar V;\zmodp)$ and 
$v\in H^{2m-1}(\Omega\bar V;\zmodp)$ be the transgressions of $x$ and $y$.

The following lemma can be viewed as an extension of Lemma~\ref{L2}.

\begin{lemma}
\label{L3}
Let $k_1\geq 1$. There exists a map \seqm{S^{4m-1}}{\bar f}{\Omega W} such that the composite
$$
\seqmm{S^{4m-1}}{\bar f}{\Omega W}{\Omega q}{\Omega V}
$$
induces a map sending a generator $\iota_{2m-1}\in H_*(S^{4m-1};\zmodp)$ to $v^2\in H_*(\Omega V;\zmodp)$.
\end{lemma}

\begin{proof}
Let \seqm{S^{4m-2}}{\alpha}{\M{2m}{r_1}} be the attaching map for the $(4m-1)$-cell of $V$, 
and \seqm{S^{4m-2}}{\xi}{\bar W} the attaching map for the $(4m-1)$-cell of $W$. 
Notice the map \seqm{W}{q}{V} is the extension of a map \seqm{\bar W}{\bar q}{\M{2m}{r_i}} fitting in a diagram of 
cofibration sequences
\[\diagram
S^{4m-2}\rto^{\xi}\ddouble & \bar W \dto^{\bar q}\rto^{i_W} & W\dto^{q}\\
S^{4m-2}\rto^{\alpha} & \M{2m}{r_1}\rto^{i_V} & V.
\enddiagram\]
Proposition~\ref{Lexp} implies $[\xi]$ has order $p^{r_1}$ in $\pi_{4m-2}(\bar W)$, since $r_1=\max\paren{r_1,\ldots, r_k}$. 
Thus $\xi$ extends to a map \seqm{\M{4m-1}{r_1}}{\bar\xi}{\bar W}. 

Let \seqm{\M{4m-2}{r_1}}{\bar\xi'}{\Omega\bar W} denote the adjoint of $\bar\xi$. Let $u'\in H_{4m-3}(\M{4m-2}{r_1};\zmodp)$ 
and $v'\in H_{4m-2}(\M{4m-2}{r_1};\zmodp)$ be generators satisfying $\beta_r(v')=u'$. By the above diagram of
cofibrations, $\Omega\bar q\circ\bar\xi'$ restricted to $S^{4m-3}$ is the adjoint of $\alpha$, so Corollary~\ref{C1}
implies
$$
(\Omega\bar q\circ\bar\xi')_*(u')=[u,v]
$$
for some choice of our generator $u'$. 

When $m>2$, $H_{4m-2}(\Omega \M{2m}{r_1};\zmodp)$ is $1$-dimensional vector space with a basis given by $v^2$. 
Thus $(\Omega\bar q\circ\bar\xi')_*(v')=kv^2$ for some $k$, and 
$$
k[u,v]=\beta_{r_1}(kv^2)=\beta_{r_1}((\Omega\bar q\circ\bar\xi')_*(v'))
=(\Omega\bar q\circ\bar\xi')_*(\beta_{r_1}(v'))=(\Omega\bar q\circ\bar\xi')_*(u')=[u,v],
$$
so $k=1$. Therefore 
$$
(\Omega\bar q\circ\bar\xi')_*(v')=v^2.
$$ 
Consider the composite
$$
f\colon\seqmm{\M{4m-2}{r_1}}{\bar\xi'}{\Omega\bar W}{\Omega i_W}{\Omega W}.
$$
Now $\Omega q\circ f$ is homotopic to $\Omega i_V \circ\Omega\bar q\circ\bar\xi'$, and since 
$H_*(\Omega V;\zmodp)\cong T(u)\otimes T(v)$ such that $H$-map \seqm{\Omega\M{2m}{r_1}}{\Omega i_V}{\Omega V} induces a map on mod-$p$ homology
modelled by the algebra map \seqm{T(u,v)}{}{T(u)\otimes T(v)},
$$
(\Omega q\circ f)_*(v')= (\Omega i_V)_*\circ(\Omega\bar q\circ\bar\xi')_*(v')=(\Omega i_V)_*(v^2)=v^2.
$$ 
Notice $f$ factors through the quotient map \seqm{\M{4m-2}{r_1}}{}{S^{4m-2}}, as the restriction of $f$ to the bottom sphere 
$S^{4m-3}$ is null homotopic, since it is the adjoint of the (null homotopic) composite \seqmm{S^{4m-2}}{\xi}{\bar W}{i_W}{W}. 
Thus $f$ extends to a map \seqm{S^{4m-2}}{\bar f}{\Omega W} so that $\bar f_*(\iota_{4m-2})=f_*(v')$. Therefore
$$
(\Omega q\circ \bar f)_*(\iota_{4m-2})=(\Omega q\circ f)_*(v')=v^2.
$$
This completes the proof.
\end{proof}

\begin{corollary}
\label{C3}
The map \seqm{\Omega W}{\Omega q}{\Omega V} has a right homotopy inverse. 
\end{corollary}
\begin{proof}
Assume $k_1\geq 1$ for now. By Theorem~\ref{P0} there is a Hopf algebra isomorphism
$$
H_*(\Omega V;\zmodp)\cong T(u)\otimes T(v),
$$ 
and the looped inclusion \seqm{\Omega\M{2m}{r_1}}{\Omega i_V}{\Omega V} induces a map on mod-$p$ homology
modelled by the algebra map \seqm{T(u,v)}{}{T(u)\otimes T(v)} sending $u$ to $u$ and $v$ to $v$.
Dependning on our construction of the map \seqm{W}{q}{V} at the start of this section, we can 
take an inclusion \seqm{\M{2m}{r_1}}{j}{W} such that the composite \seqmm{\M{2m}{r_1}}{j}{W}{q}{V}
is homotopic to the inclusion $i_V$. Now consider the composite
$$
\phi\colon\seqmm{\MS{2m-1}{r_1}}{h}{\Omega\M{2m}{r_1}}{\Omega j}{\Omega W}.
$$
The map $h_*$ is modelled by taking $H_*(\MS{2m-1}{r_1};\zmodp)$ isomorphically onto the left $T(u)$-submodule of 
$T(u,v)$ with basis \paren{1,v}, so $(\Omega q\circ\phi)_*$ is modelled by an isomorpism onto the left $T(u)$-submodule of 
$T(u)\otimes T(v)$ with basis \paren{1,v}.

From Lemma~\ref{L3} one has a map \seqm{S^{4m-2}}{\bar f}{\Omega W} satisfying
$(\Omega q\circ \bar f)_*(\iota_{4m-2})=v^{2}\in H_*(\Omega V;\zmodp)$.
Since $\Omega W$ is an $H$-space, by the universal property of the James construction for $\Omega S^{4m-1}$
$\bar f$ extends to an $H$-map
$$
\tilde f\colon\seqm{\Omega S^{4m-1}}{}{\Omega W},
$$
and $(\Omega q\circ\tilde f)_*$ induces an isomorphism onto the subalgebra of $H_*(\Omega V;\zmodp)$ generated by $v^2$.

One now sees that the product
$$
\seqmmm{\MS{2m-1}{r_1}\times\Omega S^{4m-1}}{\phi\times\tilde f}{\Omega W\times\Omega W}
{\Omega q\times\Omega q}{\Omega V\times\Omega V}{mult.}{\Omega V}
$$
induces an isomorphism on mod-$p$ homology, and therefore is a homotopy equivalence. 
Since $\Omega q$ is an $H$-map, this homotopy equivalence is homotopic to the composite
\begin{equation}
\label{EHomotopyEquiv}
\seqmmm{\MS{2m-1}{r_1}\times\Omega S^{4m-1}}{\phi\times\tilde f}{\Omega W\times\Omega W}
{mult.}{\Omega W}{\Omega q}{\Omega V},
\end{equation}
and so $\Omega q$ has a right homotopy inverse. 

Now consider $k_1=0$. 
Let \seqm{S^{2m-1}\vee S^{2m}}{}{W} be the inclusion inducing an isomorphism 
on mod-$p$ homology onto the subgroups generated by $x_1$ and $y_1$. Then the composite
$$
i'\colon\seqmm{S^{2m-1}\vee S^{2m}}{}{W}{q}{V}
$$
is the inclusion of the $(4m-2)$-skeleton of $V$. 
As we saw in the proof of Lemma~\ref{T1b}, 
$\Omega i'$ has a right homotopy inverse, and we are done.

\end{proof}

It will be very covenient to make a change in basis in the proof of the next lemma.  
Depending on which of the four cases the matrix $A_{z^*}$ satisfies, as described at the start of this section, 
we change our basis \zmodp\paren{x_1,x_2,x_3,\ldots x_k} and \zmodp\paren{y_1,y_2,y_3,\ldots y_k} to 
\zmodp\paren{a_1,a_2,a_3,\ldots a_k} and \zmodp\paren{b_1,b_2,b_3,\ldots b_k} 
so that for $i\geq 2$ the following are satisfied:
$$
q_*(a_1)=x, q_*(b_1)=y; \beta_{r_1}(b_1)=a_1\mbox{ if }1\leq k_1;
$$
$$
q_*(a_i)=0, q_*(b_i)=0; \beta_{r_i}(b_i)=a_i\mbox{ if }i\leq k_1;
$$ 
$$
a_1^*b_1^*=cz^*\in H^{4m-1}(W;\zmodp);
$$ 
for some integer $c$ prime to $p$. 

Since $q_*(x_i)=0$ and $q_*(y_i)=0$ when $i>2$, we can set $a_i=x_i$ and $b_i=y_i$. 
When the first case is satisfied, or when $k_1=0$, we may leave our previous basis as it was. 
For parts $(a)$, $(b)$, and $(c)$ of the second case, by inspection we can set: 
$a_1=x_2$, $b_1=y_2$, $a_2=x_1$, and $b_2=y_1$; 
$a_1=\frac{1}{2}(x_1+x_2)$, $b_1=\frac{1}{2}(y_1+y_2)$, $a_2=x_1-x_2$, and $b_2=y_1-y_2$; 
$a_1=x_1$, $b_1=y_1+y_2$, $a_2=x_1-x_2$, and $b_2=-y_2$, respectively.

Let $F$ be the homotopy fiber of \seqm{W}{q}{V},
and
\begin{equation}
\label{EPrincipalFib}
\seqmm{\Omega V}{\delta}{F}{}{W} 
\end{equation}
be the induced principal homotopy fibration sequence,
meaning there exists a left action 
$$
\mu\colon\seqm{\Omega V\times F}{}{F}
$$
such that the following diagram commutes up to homotopy
\begin{equation}
\label{DAction}
\diagram
\Omega V\times \Omega V\rto^{\mathbbm{1}\times\delta}\dto^{mult.} & \Omega V\times F\dto^{\mu}\\
\Omega V\rto^{\delta} & F. 
\enddiagram
\end{equation}

\begin{lemma}
\label{L4}
There is isomorphism of left $H_{*}(\Omega V;\zmodp)$-modules
$$
H_*(F;\zmodp)\cong \zmodp\paren{a_i,b_i|2\leq i\leq k}\otimes H_*(\Omega V;\zmodp),
$$
where $|a_i|=2m-1$, $|b_i|=2m$, $\beta_{r_i}(b_i)=a_i$ when $i\leq k_1$, 
and the left action of $H_{*}(\Omega V;\zmodp)$ is induced by $\mu$. 
\end{lemma}
\begin{proof}
Recall the Hopf algebra isomorphism $H_*(\Omega V;\zmodp)\cong T(u)\otimes T(v)$ from Theorem~\ref{P0}.
That is, $H_*(\Omega V;\zmodp)$ is the polynomial algebra over $\zmodp$ generated by $u$ and $v$,
where $u,v\in H_{*}(\Omega V;\zmodp)$ are the transgressions of $x,y\in H_{*}(V;\zmodp)$
in the mod-$p$ homology Serre spectral sequence for the path fibration of $V$.

The mod-$p$ homology spectral sequence $E$ for the principal homotopy fibration 
\seqmm{\Omega V}{\delta}{F}{}{W} is a spectral sequence of left $H_*(\Omega V;\zmodp)$-modules, with  
$$
E^{2}_{*,*}\cong H_*(W;\zmodp)\otimes H_*(\Omega V;\zmodp),
$$
and left action induced by $\mu$.
Differentials commute with the left action of $H_*(\Omega V;\zmodp)$, that is 
$d^n(f\otimes gh)=(1\otimes g)d^n(f\otimes h)$ whenever it makes sense.

Notice the generators $a_i\otimes 1,b_i\otimes 1\in E^{2}_{*,0}$ are transgressive
since $\bar W$ is a suspension. 
Since $q_*(a_i)=q_*(x_i)=0$ and $q_*(b_i)=q_*(y_i)=0$ for $i>2$, 
and likewise $q_*(a_2)=0$ and $q_*(b_2)=0$,
$$
d^{2m-1}(a_i\otimes g)=0, d^{2m}(b_i\otimes g)=0
$$
for every $g\in H_*(\Omega V;\zmodp)$ and $i>1$. 
Since $q_*(a_1)=x$ and $q_*(b_1)=y$, and since $u,v\in H_{*}(\Omega V;\zmodp)$ are
the transgressions of $x,y\in H_{*}(V;\zmodp)$, 
$$
d^{2m-1}(a_1\otimes 1)=1\otimes u, d^{2m}(b_1\otimes 1)=1\otimes v.
$$
Notice $H_*(\Omega V;\zmodp)$ is generated by elements of the form $v^l$ or $ug$ 
for monomials $g\in H_*(\Omega V;\zmodp)$. Since
$$
d^{2m-1}(a_1\otimes g)= (1\otimes g)d^{2m-1}(a_1\otimes 1)= (1\otimes g)(1\otimes u) = 1\otimes gu,
$$ 
no element $1\otimes gu$ survives to $E^{2m}_{0,*}$. 
Likewise, elements $1\otimes v^l$ do not survive to $E^{2m+1}_{0,*}$ since
$$
d^{2m}(b_1\otimes v^{l-1}) = (1\otimes v^{l-1})d^{2m}(b_1\otimes 1) = (1\otimes v^{l-1})(1\otimes v) = 1\otimes v^l.
$$
So we see that no element in $E^{2}_{0,*}$ survives to $E^{\infty}_{0,*}$.

Now consider those elements of the form $b_1\otimes gu$. 
Since $1\otimes gvu=1\otimes guv$ does not survive to $E^{2m}_{0,*}$, we have
$$
d^{2m}(b_1\otimes gu) = 0.
$$

Finally consider elements of the form $z\otimes g$.
Note $d^{i}(z\otimes g)=0$ for $i<2m-1$ for degree reasons.
Let the integers $c_1,\ldots, c_k$ modulo $p$ be such that $b_i^*a_1^*=c_i z^*$. 
As mentioned before, $c_1=c$ is nonzero. 
Dualizing to the mod-$p$ cohomology spectral sequence associated with our homotopy fibration, 
for each $i$,
$$
d_{2m-1}(b_i^*\otimes u^*)=d_{2m-1}(b_i^*\otimes 1)(1\otimes u^*)+(-1)^{|b_i|}(b_i^*\otimes 1)d_{2m-1}(1\otimes u^*)
=(b_i^*\otimes 1)(a_1^*\otimes 1)= c_i(z^*\otimes 1).
$$
Dualizing in the other direction, we have (denoted by $\zeta$)
$$
\zeta=d^{2m-1}(z\otimes 1)=\csum{i=1}{k}{c_i(b_i\otimes u)},
$$
so for each $g\in H_*(\Omega V;\zmodp)$
$$
d^{2m-1}(z\otimes g)=\csum{i=1}{k}{c_i(b_i\otimes gu)} = (1\otimes g)\zeta.
$$
Therefore since $c_1\neq 0$, $b_1\otimes gu$ is identified with a linear combination of 
elements $b_i\otimes gu$ for $i>1$ in $E^{2m}_{2m,*}$.

All the differentials in this spectral sequence have now been computed.
Summarizing all the above information, we see that
no element in $E^{2}_{0,*}$ survives to $E^{\infty}_{0,*}$, 
elements of the form $a_i\otimes g$ for $i>1$ generate the kernel of 
$\seqm{E^{2m-1}_{2m-1,*}}{d^{2m-1}}{E^{2m-1}_{0,*}}$,
and elements of the form $b_i\otimes g$ for $i>1$ and $b_1\otimes gu$ generate the kernel of 
$\seqm{E^{2m}_{2m,*}}{d^{2m}}{E^{2m}_{0,*}}$.
But $b_1\otimes gu$ is identified with a linear combination of 
$b_i\otimes gu$ for $i>1$ in $E^{2m}_{2m,*}$,
so $b_i\otimes g$ for $i>1$ generate this kernel.
Since no nonzero linear combination of elements of the form
$a_i\otimes g$ and $b_i\otimes g$ for $i>1$ is in the image of a differential,
one sees $E^{\infty}_{*,*}$ is generated by a basis of elements
$a_i\otimes g$ and $b_i\otimes g$ for $i>1$ and monomials $g\in H_*(\Omega V;\zmodp)$.
The result follows.    

\end{proof}

Let
\begin{equation}
\label{EWedge}
J=\cvee{i=2}{k_1}\M{2m}{r_i}\vee\cvee{i=k_1+1}{k}{(S^{2m-1}\vee S^{2m})}.
\end{equation}
This is the $(4m-2)$-skeleton $\bar W$ of $W$ with $\M{2m}{r_1}$ quotiented out. 
Our previous work amounts to the following:
\medskip
\begin{theorem}
\label{T2}\mbox{}
\begin{romanlist}
\item If $k_1\geq 1$ there is a homotopy equivalence
$$\Omega W\simeq\MS{2m-1}{r_1}\times\Omega S^{4m-1}\times\Omega\bracket{J\vee (J\wedge (\MS{2m-1}{r_1}\times\Omega S^{4m-1}))},$$
where the right-hand space is taken to be a weak product.
\item Similarly, if $k_1=0$ there is a homotopy equivalence
$$\Omega W\simeq\Omega S^{2m-1}\times\Omega S^{2m}\times\Omega\bracket{J\vee (J\wedge (\Omega S^{2m-1}\times\Omega S^{2m}))}.$$
\end{romanlist}
\end{theorem}
\begin{proof}[Proof of part $(i)$]
By Corollary~\ref{C3} \seqm{\Omega W}{\Omega q}{\Omega V} has a right homotopy inverse, 
so the homotopy fibration sequence \seqmm{\Omega F}{}{\Omega W}{\Omega q}{\Omega V} is split.
Therefore
$$
\Omega W\simeq\Omega V\times\Omega F\simeq\MS{2m-1}{r_1}\times \Omega S^{4m-1}\times\Omega F,
$$
where $\Omega V\simeq\MS{2m-1}{r_1}\times\Omega S^{4m-1}$ by Theorem~\ref{T1}. 

Let \seqm{\Omega V}{s}{\Omega W} be a right homotopy inverse of $\Omega q$, and
\seqm{\Omega V}{\delta}{F} be the connecting map in the homotopy fibration sequence~(\ref{EPrincipalFib}). 
Since $\delta\circ\Omega q$ is null homotopic, we have $\delta\simeq\delta\circ\Omega q\circ s\simeq *$, 
so $\delta$ is null homotopic as well. Now by Lemma~\ref{L4} the $2m$-skeleton of $F$ is 
the wedge sum $J$ in equation~(\ref{EWedge}).

Define the composite
$$
\lambda:\seqmm{\Omega V\times J}{\mathbbm{1}\times j}{\Omega V\times F}{\mu}{F},
$$
where $j$ is the inclusion of the $2m$-skeleton. Notice the composite 
$$
\seqmm{\Omega V\times *}{\mathbbm{1}\times *}{\Omega V\times J}{\lambda}{F}
$$
is null homotopic, as it is homotopic to $\delta\simeq *$ by diagram~(\ref{DAction}). 
Therefore one obtains an extension $\bar\lambda$ of $\lambda$ in the following homotopy commutative diagram
\[\diagram
\Omega V\times *\rto^-{\mathbbm{1}\times *}\ar@{.>}[dr]^{\delta\simeq *}
& \Omega V\times J\rto^{}\dto^{\lambda}
& \Omega V\ltimes J\ar@{.>}[dl]^{\bar\lambda}\\
& F
\enddiagram\]
where the \emph{half-smash product} $\Omega V\ltimes J$ is by definiton the 
cofiber of the inclusion $\Omega V\times*\subset \Omega V\times J$. By Lemma~\ref{L4} 
$$
H_*(F;\zmodp)\cong \bar H_*(J;\zmodp)\otimes H_*(\Omega V;\zmodp) \cong H_*(\Omega V\ltimes J;\zmodp).
$$
Observe that $\lambda$ restricts to an isomorphism of the submodule 
$$
\bar H_*(J;\zmodp)\otimes H_*(\Omega V;\zmodp)\subseteq H_*(\Omega V\times J;\zmodp)
$$ 
onto $H_*(F;\zmodp)$, so $\bar\lambda$ induces an isomorphism on mod-$p$ homology. Therefore
$$
F\simeq \Omega V\ltimes J.
$$ 
Now applying the well known general splitting of half-smash products
$$
B\ltimes (\Sigma A)=(\Sigma A)\rtimes B\simeq (\Sigma A)\vee(\Sigma A\wedge B),
$$
one obtains
\begin{align*}
F & \simeq \Omega V\ltimes J\\
& \simeq J\vee (J\wedge \Omega V)\\
& \simeq J\vee (J\wedge(\MS{2m-1}{r}\times\Omega S^{4m-1})),
\end{align*}
and we are done
\end{proof}
\begin{proof}[Proof of part $(ii)$]
The proof is identical to that of part $(i)$, 
except with Theorem~\ref{T1b} used in place of Theorem~\ref{T1}.
\end{proof}

\begin{proof}[Proof of Theorem~\ref{Main} and Theorem~\ref{Main0} part $(i)$]
For each of the homotopy equivalences in Theorem~\ref{T2}
the homotopy type of the right-hand weak product is uniquely determined by the integers $k$ and $k_1$,
and the integers $r_1,\ldots,r_{k_1}$. The ordering is arbitrary, 
but we selected it so that $r_1=\max\{r_1,\ldots,r_{k_1}\}$ when $k_1>0$.
As is clearly seen in Equation~(\ref{ESplit}), 
these integers correspond uniquely to the homotopy type of $\bar W$.
Therefore the homotopy types of the right-hand weak products in Theorem~\ref{T2}
are uniquely determined by the homotopy type of $\bar W$.

It is clear that any two $[W],[W']\in\mathcal{T}^{p}_{k,2m}$ satisfy 
conditions $(1)$ and $(2)$ in Theorem~\ref{Main0} if and only if $\bar W\simeq \bar W'$.
The result follows by application of Theorem~\ref{T2}. 
\end{proof}

\section{The Integral Case}
\label{Integral}

In this section spaces are not assumed to be localized. 
When we say that the localization of map or space at a prime $p$ \emph{is} another map or space,
we will mean that they are the same at least up to homotopy equivalence. 

Let $\mathcal{P}=\paren{p_1,p_2,\ldots}$ be the set of all prime numbers,
and $M$ be a manifold as in part $(ii)$ of Theorem~\ref{Main0}.
The uniqueness up to homotopy type of Moore spaces 
implies the $(4m-2)$-skeleton $\bar M$ of $M$ splits as a finite wedge of Moore spaces
$$
\bar M\simeq\cvee{i}{}{\bracket{\cvee{j}{}{P^{2m}(q_i^{r_{i,j}})}}},
$$
where $\mathcal{Q}=\paren{q_1,q_2,\ldots}\subset\paren{p_1,p_2,\ldots}$ is some subset of odd prime numbers.
We may as well assume the homotopy type of $M$ (which might now longer be a manifold) 
has been selected so the $\bar M$ is homeomorphic to the above wedge of Moore spaces.

When localized at some $p$, mod-$q$ Moore spaces are contractible whenever $q$ prime to $p$.
Then $\bar M_{(q_i)}$ is homotopy equivalent to a wedge of the mod-$q_i$ Moore spaces in the above splitting.
On the other hand, $\bar M_{(p)}$ is contractible when localized at any $p\in\mathcal{P}-\mathcal{Q}$, 
which implies $M_{(p)}\simeq S^{4m-1}_{(p)}$.

We will need to lift some of the $p$-local maps constructed in the previous section
to ones that are integral.
This is perhaps best done by following through the same constructions,
all the while keeping in mind we are no longer localized.  
First, using the recipe for the construction of the map \seqm{W}{q}{V} in~(\ref{EMainMap}),
we can collapse or fold the Moore spaces in the $(4m-2)$-skeleton $\bar M$ to obtain a map
$$
\rho\colon\seqm{M}{}{N}
$$  
whose localization at each prime $p\in\mathcal{Q}$ is the map $q$, 
and whose localization at each $p\in\mathcal{P}-\mathcal{Q}$ is a homotopy equivalence
\seqm{M\simeq S^{4m-1}_{(p)}}{\simeq}{S^{4m-1}_{(p)}\simeq N}.
Then the $(4m-2)$-skeleton of $N$ is a finite wedge
$$
\bar N\simeq\cvee{i}{}{P^{2m}(q_{i}^{s_{i}})},
$$
where $s_i=\max_j\paren{r_{i,j}}$.

Let 
$Q=\cprod{i}{}{S^{2m-1}\{q_{i}^{s_{i}}\}}$
and 
$Q'=\cprod{i}{}{\Omega P^{2m}(q_{i}^{s_{i}})}$.
The $p$-local map $h$ in Section~\ref{RankOne}
can be lifted to an integral map $h'$ that is a choice
of lift in the digram of homotopy fibrations sequence 
\[\diagram
\MS{2m-1}{r}\rto^{}\dto^{h'} & S^{2m-1}\rto^{}\dto^{} & S^{2m-1}\dto^{\subset}\\
\Omega P^{2m}(p^{r})\rto^{} & \ast\rto^{} & P^{2m}(p^{r}),
\enddiagram\] 
for this is how $h$ was constructed in~\cite{CMN}.
Taking products of these maps defines the obvious map \seqm{\Omega Q}{}{\Omega Q'}.
Using the Hilton-Milnor homotopy decomposition of $\Omega\bar N$
gives a map \seqm{Q'}{}{\Omega\bar N} that is a left homotopy inverse of the 
canonical looped inclusion \seqm{\Omega\bar N}{}{Q'}. Now we define the composite
$$
\eta\colon\seqmmm{Q}{}{Q'}{}{\Omega\bar N}{}{\Omega M},
$$
where the last map is the looped inclusion.

Next, construct a map 
$$
\tilde g\colon\seqm{\Omega S^{4m-1}}{}{\Omega M},
$$
which is the integral analogue of the map $\tilde f$ in the proof of Corollary~\ref{C3}.
The construction begins along the same lines as that of the map $\bar f$ in Lemma~\ref{L3}.
First take the attaching map \seqm{S^{4m-2}}{\xi}{\bar M} for the $(4m-1)$-cell of $M$.
By localizing $\bar M$ at each prime $p$ and applying Proposition~\ref{Lexp},
$[\xi]$ must be of order $s=\prod_i {q_i}^{s_i}$, so $\xi$ factors through a map 
\seqm{P^{4m-1}(s)}{\bar\xi}{\bar M}.
We then let \seqm{P^{4m-2}(s)}{\bar\xi'}{\Omega\bar M} be the adjoint of $\bar\xi$.
The restriction of the map 
$$
g\colon\seqmm{P^{4m-2}(s)}{\bar\xi'}{\Omega\bar M}{}{\Omega M}
$$
to the bottom sphere $S^{4m-3}$ is null homotopic since it is the adjoint
of \seqmm{S^{4m-2}}{\xi}{\bar M}{}{M}, which itself is null homotopic.
Then $g$ factors through a map 
$$
\bar g\colon\seqm{S^{4m-2}}{}{\Omega M},
$$
By the universal property of the James construction for $\Omega S^{4m-1}$,
$\bar g$ extends to the map $\tilde g$ detailed above.

The following is well known (see~\cite{MNT1} for example):

\begin{proposition}
\label{LPequiv}
A map \seqm{X}{}{Y} of finite type $1$-connected $CW$-complexes is a homotopy equivalence 
if and only if it induces a homotopy equivalence localized at each prime $p$.~$\qqed$
\end{proposition}

Using this, and our preceeding work localized at odd primes $p$,
we prove the integral classification in part $(ii)$ of Theorem~\ref{Main0}.

\begin{proof}[Proof of Theorem~\ref{Main0} part $(ii)$]

Combine our maps to obtain the composite
$$
\psi\colon\seqmmm{Q\times\Omega S^{4m-1}}{\eta\times\tilde g}{\Omega M\times \Omega M}{mult.}{\Omega M}{\Omega\rho}{\Omega N}.
$$ 
Notice a $S^{2m}\{q^r\}$ factor in $Q$ has trivial mod-$p$ homology when $q$ is prime to $p$,
and thus is contractible when localized at $p$. 
The maps $\eta$ and $\tilde g$ localize at $p\in\mathcal{Q}$ 
to the corresponding maps $\phi$ and $\tilde f$ in composite the composite (\ref{EHomotopyEquiv}),   
so $\psi_{(p)}$ is the composite (\ref{EHomotopyEquiv}).
On the other hand, $\psi_{(p)}$ reduces to a homotopy equivalence \seqm{\Omega S^{4m-1}_{(p)}}{\simeq}{\Omega N_{(p)}}    
when localized at $p\in\mathcal{P}-\mathcal{Q}$.
Now applying Lemma~\ref{LPequiv}, $\psi$ is itself a homotopy equivalence. 

Let $G$ be the homotopy fiber of \seqm{M}{\rho}{N}. 
The homotopy fibration sequence 
$$
\seqmm{\Omega G}{}{\Omega M}{\Omega\rho}{\Omega N} 
$$
therefore has a homotopy cross-section given by the homotopy equivalence $\psi$,
and as such there are homotopy equivalences  
$$
\Omega G\times Q\times\Omega S^{4m-1}\simeq \Omega G\times \Omega N \simeq \Omega M.
$$

The homotopy type of $Q$ clearly depends only on the homotopy type of the $(4m-2)$-skeleton $\bar M$. 
To complete the proof, we need to show that the homotopy type of $\Omega G$ 
also depends only on that of $\bar M$. 

Let $I$ be the quotient $\bar M / \bar N$. Localized at $p\in\mathcal{Q}$, 
$I_{(p)}$ is homotopy equivalent to the wedge $J$ in Theorem~\ref{T2}.
The localization $G_{(p)}$ is contractible when $p\in\mathcal{P}-\mathcal{Q}$, 
since $\rho$ is a homotopy equivalence in this case. 
When $p\in\mathcal{Q}$, $G_{(p)}$ is the homotopy fiber $F$ in Theorem~\ref{T2}, 
and we have shown that the $2m$-skeleton of $F$ is $J\simeq I_{(p)}$. 
Thus $I$ is the $2m$-skeleton of $G$.
We can now carry forward with a construction similar to the one in the proof of Theorem~\ref{T2},
and produce a map \seqm{\Omega N\ltimes I}{\bar\gamma}{G} whose localization $\bar\gamma_{(p)}$ at $p\in\mathcal{Q}$  
is the homotopy equivalence \seqm{\Omega V\ltimes J}{\bar\lambda}{F} in the proof thereof. 
Localized at primes $p\in\mathcal{Q}$, $\bar\gamma_{(p)}$ is also a homotopy equivalence,
since here both $G_{(p)}$ and $(\Omega N\ltimes I)_{(p)}\simeq I_{(p)}\vee I_{(p)}\wedge \Omega N_{(p)}$
are contractible. Lemma~\ref{LPequiv} now implies
$$
G\simeq \Omega V\ltimes I.
$$
The homotopy type of $\bar N$ is depends only on the homotopy type of $\bar M$, 
and so the same applies for $I$, and consequently for $G$ and $\Omega G$ as well.

\end{proof}

\bibliographystyle{amsplain}
\bibliography{mybibliography}

\end{document}